\documentclass[10pt]{amsart}
\usepackage{graphicx}
\usepackage{amscd}
\usepackage{amsmath}
\usepackage{amsthm}
\usepackage{amsfonts}
\usepackage{amssymb}
\usepackage{mathrsfs}
\usepackage{bm}
\usepackage{enumerate}
\usepackage{amsrefs}
\usepackage{xcolor}
\usepackage[colorlinks, citecolor=blue, linkcolor=red, pdfstartview=FitB]{hyperref}

\usepackage{marginnote}	
\usepackage{soul} 			

\newcommand{\bB}{{\mathbb{B}}}
\newcommand{\bC}{{\mathbb{C}}}
\newcommand{\bD}{{\mathbb{D}}}

\newcommand{\bM}{{\mathbb{M}}}
\newcommand{\bN}{{\mathbb{N}}}

\newcommand{\bR}{{\mathbb{R}}}
\newcommand{\bS}{{\mathbb{S}}}
\newcommand{\bT}{{\mathbb{T}}}

  \newcommand{\A}{{\mathcal{A}}}

  \newcommand{\E}{{\mathcal{E}}}
  \newcommand{\F}{{\mathcal{F}}}
  \newcommand{\G}{{\mathcal{G}}}
\renewcommand{\H}{{\mathcal{H}}}
  
  \newcommand{\J}{{\mathcal{J}}}

  \newcommand{\M}{{\mathcal{M}}}
  \newcommand{\N}{{\mathcal{N}}}
\renewcommand{\O}{{\mathcal{O}}}
\renewcommand{\P}{{\mathcal{P}}}
  
  \newcommand{\R}{{\mathcal{R}}}
\renewcommand{\S}{{\mathcal{S}}}

  \newcommand{\Z}{{\mathcal{Z}}}

\newcommand{\fa}{{\mathfrak{a}}}
\newcommand{\fB}{{\mathfrak{B}}}

\newcommand{\fC}{{\mathfrak{C}}}
\newcommand{\fc}{{\mathfrak{c}}}

\newcommand{\fJ}{{\mathfrak{J}}}
\newcommand{\fK}{{\mathfrak{K}}}

\newcommand{\fT}{{\mathfrak{T}}}

\newcommand{\fW}{{\mathfrak{W}}}

\newcommand{\fz}{{\mathfrak{z}}}

\newcommand{\rC}{\mathrm{C}}

\newcommand{\eps}{\varepsilon}
\renewcommand{\phi}{\varphi}

\newcommand{\upchi}{{\raise.35ex\hbox{$\chi$}}}

\newcommand{\lip}{\langle}
\newcommand{\rip}{\rangle}
\newcommand{\ip}[1]{\lip #1 \rip}

\newcommand{\ol}{\overline}

\newcommand{\nin}{\notin}
\newcommand{\of}{\circ}

\newcommand{\qand}{\quad\text{and}\quad}

\newcommand{\ran}{\operatorname{ran}}

\newcommand{\spn}{\operatorname{span}}
\newcommand{\supp}{\operatorname{supp}}

\newcommand{\fb}{\mathfrak{b}}

\newcommand{\Ta}{\operatorname{Ta}}

\newtheorem{lemma}{Lemma}[section]
\newtheorem{theorem}[lemma]{Theorem}

\newtheorem{corollary}[lemma]{Corollary}

\theoremstyle{definition}

\newtheorem{example}{Example}

\begin{document}
\author{Rapha\"el Clou\^atre}

\address{Department of Mathematics, University of Manitoba, Winnipeg, Manitoba, Canada R3T 2N2}

\email{raphael.clouatre@umanitoba.ca\vspace{-2ex}}
\thanks{RC was partially supported by an NSERC Discovery Grant. EJT was partially supported by a PIMS postdoctoral fellowship.}
\author{Edward J. Timko}
\email{edward.timko@umanitoba.ca\vspace{-2ex}}


\begin{abstract}
The main objects under study are quotients of multiplier algebras of certain complete Nevanlinna--Pick spaces, examples of which include the Drury--Arveson space on the ball and the Dirichlet space on the disc. We are particularly interested in the non-commutative Choquet boundaries for these quotients. Arveson's notion of hyperrigidity is shown to be detectable through the essential normality of some natural multiplication operators, thus extending previously known results on the Arveson--Douglas conjecture. We also highlight how the non-commutative Choquet boundaries of these quotients are intertwined with their Gelfand transforms being completely isometric. Finally, we isolate analytic and topological conditions on the so-called supports of the underlying ideals that clarify the nature of the non-commutative Choquet boundaries.

\end{abstract}

\title[Gelfand transforms of Nevanlinna--Pick quotients]{Gelfand transforms and boundary representations of complete Nevanlinna--Pick quotients}
\date{\today}
\maketitle

\section{Introduction}\label{S:intro}

The symbiotic relationship between analytic function theory on the open unit disc $\bD\subset \bC$ and operator theory has long been fruitfully exploited.  The standard setting is that of the Hardy space $H_1^2$ consisting of those analytic functions on $\bD$ with square summable Taylor coefficients at the origin, and the main actor is the operator $M_z:H_1^2\to H_1^2$ of multiplication by the variable $z$. Given an inner function $\theta$, we can consider the so-called ``model operator'', obtained as the following compression
\[
 Z_\theta=P_{(\theta H_1^2)^\perp} M_z|_{(\theta H_1^2)^\perp}.
\]
This operator and the various operator algebras related to it contain a striking amount of information about the function $\theta$. One key to unraveling this information is an additional piece of structure that the Hardy space enjoys: the Szeg\"o kernel 
\[
 k_1(z,w)=(1-z\ol{w})^{-1}, \quad z,w\in \bD
\]
is its reproducing kernel.

The potential lying within these concepts and their interactions was uncovered early on by Sarason in his seminal paper \cite{sarason1967}, wherein it was shown how a complete description of the operators commuting with $Z_\theta$ is equivalent to several classical interpolation results, once the function $\theta$ is chosen appropriately. In this instance, operator theory clarifies a function theory problem. 
In the reverse direction, operator theoretic information can be extracted from the fine function theoretic properties of the symbol $\theta$. Indeed, the model operators $Z_\theta$, and operator-valued versions thereof, are the building blocks for a very robust classification theory of Hilbert space contractions (see \cite{nagy2010},\cite{bercovici1988}). 

In the aforementioned developments a unifying central figure is what is now called the Commutant Lifting Theorem  \cite{Nagy1968}, which identifies the commutant of $Z_\theta$ as the weak-$*$ closed unital operator algebra that it generates. This operator algebra can also be regarded as the quotient $\M_1/\theta \M_1$, where $\M_1$ denotes the multiplier algebra of $H^2_1$. We then see that the quotient algebra $\M_1/\theta \M_1$ carries information that is relevant both from the function theoretic and the operator theoretic points of view. Operator algebraic properties of these quotients were further leveraged in \cite{clouatre2015jordan}  to sharpen some classification results. More precisely, a key point there was whether or not the identity representation lies in the non-commutative Choquet boundary. An earlier foray into such questions can be found in Arveson's work in \cite{arveson1969} and \cite{arveson1972}, which undertook a careful study of the norm-closed subalgebra of $\M_1/\theta \M_1$ generated by $Z_\theta$.

Although the foregoing discussion is set in the one-variable world of the unit disc, there is an analogous multivariate version of the story that takes places on the open unit ball $\bB_d\subset \bC^d$, where $d$ is some positive integer. This higher-dimensional setting has been a fertile ground for many years for research in both operator theory and function theory, and it continues to be to this day. Here, the classical Hardy space $H_1^2$ is replaced with an appropriate generalization, called the Drury--Arveson space and denoted by  $H^2_d$. The reproducing kernel of $H^2_d$ is
\[
 k_d(z,w)=(1-\langle z,w\rangle_{\bC^d})^{-1}, \quad z,w\in \bB_d
\]
and the multiplier algebra is denoted by $\M_d$. The intense research activity surrounding this space stems partly from the fact that it is universal (in a precise technical sense) for both the purposes of function theory \cite{AM2000} and of operator theory \cite{arveson1998}. 

A commutant lifting theorem is known in this context as well \cite{BTV2001}. Furthermore, mirroring the univariate developments, there is a robust yet still developing corresponding classification theory for commuting row contractions using weak-$*$ closed ideals $\fa\subset \M_d$ along with the associated $d$-tuples
\[
 Z_\fa=(P_{[\fa H^2_d]^\perp} M_{z_1}|_{[\fa H^2_d]^\perp},\ldots,P_{[\fa H^2_d]^\perp}M_{z_d}|_{[\fa H^2_d]^\perp})
\]
as building blocks (see \cite{muller1993},\cite{arveson1998},\cite{popescu2006},\cite{CT2019cyclic},\cite{CT2019interp} and the references therein). For these reasons, it is desirable to understand the structure of the quotients $\M_d/\fa$, as well as that of the norm-closed subalgebras generated therein by $Z_\fa$. Once again, the precise question we will explore is whether their identity representations belong to their respective non-commutative Choquet boundaries.

We note that there are some important qualitative differences between the univariate and the multivariate worlds. For instance, it is readily seen that when $d=1$, the model operator $Z_\theta$ is always essentially unitary, in the sense that
\[
 I-Z_\theta^*Z_\theta \qand I-Z_\theta Z_\theta^*
\]
are compact \cite[Theorem 3.4.2]{arveson1969}. The corresponding statement when $d>1$ is not obvious; whether or not it even holds is the content of the celebrated Arveson--Douglas essential normality conjecture (see \cite{englis2015} and \cite{DTY2016} for some recent progress on this problem). Interestingly, our methodology here is in some ways connected to this subtle question, as was first realized in \cite{kennedyshalit2015}.

Throughout the paper, we work in a more general context, one which attempts to distill the salient distinguishing feature of the Drury--Arveson space. Accordingly, we will be considering Hilbert spaces of analytic functions on the ball $\bB_d$ with a complete Nevanlinna--Pick reproducing kernel that is unitarily invariant, along with their multiplier algebras and their quotients. As we explain below, there are many important examples of such spaces, such as the Drury--Arveson space and the Dirichlet space on the disc. The impetus to treat these two spaces (and many more) simultaneously is provided by recent advances in the connection between their function theory and the associated operator theory \cite{BHM2018},\cite{CH2018},\cite{CT2019spectrum} building on ideas from \cite{agler1982},\cite{ambrozie2002},\cite{GHX04}.  With that being said, we emphasize that the setting of the Drury--Arveson space and quotients of its multiplier algebra were the original motivation for the work undertaken in this paper, and many of our results are already new in this concrete setting.

We now describe the organization of the paper and state our main results.  In Section \ref{S:OA}, we gather some background material on operator algebras, boundary representations and their $\rC^*$-envelopes. The main results of this section are Theorem \ref{T:GTC*env} and Corollary \ref{C:GTC*env}, which highlight a connection between the identity representation enjoying a certain unique extension property and the Gelfand transform of a commutative unital operator algebra being completely isometric. 

In Section \ref{S:kernels}, we define what we call maximal regular unitarily invariant reproducing kernel Hilbert spaces with the complete Nevanlinna--Pick property, along with some natural algebras of multipliers associated with them. We also prove some preliminary facts, many of which are undoubtedly known to experts but do not seem to appear explicitly in the literature. 

Let $\H$ be such a space on the unit ball $\bB_d$ and let $\M(\H)$ denote its multiplier algebra. Fix some weak-$*$ closed ideal $\fa\subset \M(\H)$ and let $\H_\fa=\H\ominus [\fa \H]$. We define a weak-$*$ closed, commutative unital operator algebra on $\H_\fa$ as
\[
 \M_\fa=\{P_{\H_\fa}M_\psi|_{\H_\fa}:\psi\in \M(\H)\}.
\]
We also let $\A_\fa\subset \M_\fa$ denote the norm-closed subalgebra generated by the compressed coordinate multipliers $\fz_1,\ldots,\fz_d$, where
\[
 \fz_j=P_{\H_\fa}M_{z_j}|_{\H_\fa},\quad 1\leq j\leq d.
\]
Let $\Delta(\A_\fa)$ denote the space of characters of $\A_\fa$ and let $g:\A_\fa\to C(\Delta(\A_\fa))$ be the Gelfand transform, so that
\[
(g(b))(\chi)=\chi(b), \quad b\in \A_\fa, \chi\in \Delta(\A_\fa).
\]
Our investigation throughout is predicated on determining whether the identity representation is a boundary representation for the algebras $\A_\fa$ and $\M_\fa$. Our analysis bifurcates into two strands, one for $\A_\fa$ and another for $\M_\fa$. The need for a separate analysis of each case is not entirely unexpected, as even in the most classical situation where $\H$ is the Hardy space on the disc, the phenomena highlighted in \cite{arveson1969},\cite{arveson1972} and \cite{clouatre2015jordan} are vastly different, and the methods of proof are quite distinct. 

The knowledge that the identity representation is a boundary representation is expected have some applications in operator theory. Indeed, it could lead to sharper versions of the results of \cite{CT2019cyclic} and \cite{CT2019interp}, in the spirit of what was done in \cite{clouatre2015jordan} in the single-variable context. A more thorough examination of this prospect will be undertaken in a future paper.

In Section \ref{S:Gelfbdry}, we analyze the identity representation of $\fT_\fa=\rC^*(\A_\fa)$ and wish to determine when it is a boundary representation for $\A_\fa$. 
We start by identifying the character space of both $\A_\fa$ and $\fT_\fa$ in Theorem \ref{T:specAaTa} and use it to calculate the essential norm of elements in $\A_\fa$ in Theorem \ref{T:Aaessnorm}. Leveraging these pieces of information, we obtain the following characterization (Theorem \ref{T:unifquotientAa}).

\begin{theorem}\label{T:main1}
Let $\H$ be a  maximal regular unitarily invariant complete Nevanlinna--Pick space. Let $\fa\subset \M(\H)$ be a proper weak-$*$ closed ideal. If $\H_\fa$ has dimension greater than one, then the following statements are equivalent.
\begin{enumerate}[{\rm(i)}]

\item The Gelfand transform of $\A_\fa$ is completely isometric.

 \item The $\rC^*$-envelope of $\A_\fa$ is commutative.

\item The identity representation of $\fT_\fa$ is not a boundary representation for $\A_\fa$.
\end{enumerate}
\end{theorem}

We note here that the case where $\H$ is the Hardy space on the unit disc was investigated by Arveson, who completely characterized the ideals $\fa$ for which $\A_\fa$ has a completely isometric Gelfand transform. The details appear in Example \ref{E:HardyAaGelf} below. We show in Example \ref{E:DAAaGelf} that our general situation is more complicated. 

Arveson's notion of hyperrigidity for operator algebras is also relevant for our purposes in Section \ref{S:Gelfbdry}. Determining whether an operator algebra is hyperrigid is typically quite difficult, as it requires checking that every representation of the generated $\rC^*$-algebra has a certain unique extension property. Nevertheless, hyperrigidity of operator algebras has generated significant research activity recently (see for instance \cite{kennedyshalit2015},\cite{DK2016},\cite{CH2018},\cite{clouatre2018unp},\cite{clouatre2018lochyp},\cite{kim2019},\cite{salomon2019} and the references therein) prompted by a compelling yet still unresolved conjecture of Arveson \cite{arveson2011}. 
It is therefore noteworthy that our approach also allows us to characterize when the algebra $\A_\fa$ is hyperrigid
(see Corollary \ref{C:idbdryessnorm}).

\begin{theorem}\label{T:main2}
Let $\H$ be a maximal  regular unitarily invariant complete Nevanlinna--Pick space. Let $\fa\subset \M(\H)$ be a proper weak-$*$ closed ideal. The following statements are equivalent.
\begin{enumerate}[{\rm (i)}]
 \item The algebra $\A_\fa$ is hyperrigid in $\fT_\fa$.

 \item The $d$-tuple $(\fz_1,\ldots,\fz_d)$ is essentially normal and the Gelfand transform of $\A_\fa$ is not completely isometric.
\end{enumerate}\end{theorem}

 This theorem and its proof are inspired by \cite[Theorem 4.12]{kennedyshalit2015} which focuses on the Drury--Arveson space. 
It should be noted that the results of \cite{kennedyshalit2015, kennedyshalit2015corr} hold for more general spaces than the Drury--Arveson space, see Section 5 therein. However, as explained on page 1731 of \cite{CH2018}, the overlap with our present setting consists \emph{only} of the Drury--Arveson space.

In the second part of Section \ref{S:Gelfbdry}, we turn our attention to the whole algebra $\M_\fa$, and analyze when the identity representation of $\rC^*(\M_\fa)$ is a boundary representation for it. We obtain an analogue of Theorem \ref{T:main1} in this context (Corollary \ref{C:Maunifalg}).
Examples \ref{E:Mdcompactquotient},  \ref{E:HardyMaGelf}, \ref{E:DAMaGelf} and \ref{E:commHinf} put the various conditions therein in perspective.

In Section \ref{S:polyapprox}, we find conditions under which the identity representation of $\fT_\fa$ necessarily is a boundary representation for $\A_\fa$, or equivalently, under which the Gelfand transform of $\A_\fa$ is not completely isometric. For instance, we show that this is almost always the case when $\fa$ is a homogeneous ideal (Theorem \ref{T:homogquotientGelf}), unless $\A_\fa$ can be identified with the classical disc algebra. Another instance where this can be verified is when the reproducing kernel of the space $\H$ fails to be maximal (Theorem \ref{T:nonmax}). For the rest of the section, we focus on the maximal situation and seek to exploit some analytic boundary assumptions on the so-called support of the ideal $\fa$, which was introduced in \cite{CT2019spectrum}. Given a compact subset $K\subset \bC^d$, we say that it is an \emph{approximation set} if the polynomials are uniformly dense in the continuous functions on $K$. The following is Corollary \ref{C:idbdryPKCK2}.

\begin{theorem}\label{T:main3}
Let $\H$ be a maximal regular unitarily invariant complete Nevanli\-nna--Pick space.
Let $\fa\subset \M(\H)$ be a proper weak-$*$ closed ideal such that $\H_\fa$ has dimension greater than one and $\supp\fa\cap \bS_d$ is an approximation set. Then, the Gelfand transform of $\A_\fa$ is not completely isometric.
\end{theorem}

We recall that in one variable, a classical theorem of Lavrentiev states that if  $K\subset \bC$ is a compact subset, then the polynomials are dense in $C(K)$ if and only if $K$ has no interior and connected complement. A similar characterization of this density property in several variables appears to be a (difficult) open problem in approximation theory \cite{levenberg2006}. Some sufficient conditions are known however, and can be used to describe concrete situations where the previous theorem applies (see Corollary \ref{C:PKCK}). In Corollary \ref{C:KSMainC}, we adapt the arguments found in \cite{kennedyshalit2015corr} to sharpen Theorem \ref{T:main2} in the presence of some geometric restrictions on the zero set of the ideal $\fa$, and bring its statement in line with that found in \cite{kennedyshalit2015corr}.

Finally, in Section \ref{S:isol}, we seek to identify sufficient conditions on the ideal $\fa$ in order for the Gelfand transform of $\M_\fa$ not to be completely isometric. We show that the presence of some isolated points in the character space of $\M_\fa$ accomplishes this (Theorem \ref{T:shilovisolMa}).

\begin{theorem}\label{T:mainisol}
Let $\H$ be a  regular unitarily invariant complete Nevanlinna--Pick space. Let $\fa\subset \M(\H)$ be a proper weak-$*$ closed ideal whose support has an isolated point $\lambda\in \bB_d$. Then, the following statements hold.
\begin{enumerate}[{\rm (i)}]
\item The algebra $\M_\fa$ contains a non-zero idempotent finite-rank operator and the identity representation of $\rC^*(\M_\fa)$ is a boundary representation for $\M_\fa$. Moreover, the Gelfand transform of $\M_\fa$ is not completely isometric if $\H_\fa$ has dimension greater than one.

 \item If $\supp \fa$ consists of more than one point, then 
the Gelfand transform of $\M_\fa$ is not isometric. 
\end{enumerate}
\end{theorem}

As an application, we obtain the following (Corollary \ref{C:sequenceMa}).

\begin{theorem}\label{T:main6}
Let $\H$ be a  regular unitarily invariant complete Nevanlinna--Pick space. Let $\fa\subset \M(\H)$ be a proper weak-$*$ closed ideal with the property that its zero set has an isolated point. 
Then, the following statements are equivalent.
 \begin{enumerate}
  \item[\rm{(i)}] The ideal $\fa$ is the vanishing ideal of a single point.
  
  \item[\rm{(ii)}] The space $\H_\fa$ is one-dimensional.
  
    \item[\rm{(iii)}] The Gelfand transform of $\M_\fa$ is completely isometric.
    
    \item[\rm{(iv)}] The Gelfand transform of $\M_\fa$ is  isometric. 
 \end{enumerate}
\end{theorem}

This result is consistent with the findings of \cite[Theorem 4.2]{clouatre2015CB} for the classical Hardy space on the unit disc, in which case the statement of the previous theorem is valid for any weak-$*$ closed ideal. Whether or not this stronger statement holds in the generality discussed here calls for further investigation.


\section{Operator algebras, boundary representations and $\rC^*$-envelopes}\label{S:OA}

Let $B(\H)$ denote the $\rC^*$-algebra of bounded linear operators on some Hilbert space $\H$. By an \emph{operator algebra} we simply mean a subalgebra $\A\subset B(\H)$. We denote by $\A'$ the commutant of $\A$, so that
\[
 \A'=\{X\in B(\H):Xa=aX, \quad a\in \A\}.
\]
For each $n\geq 1$, we let $\bM_n(\A)$ denote the algebra of $n\times n$ matrices with entries in $\A$. When $\A=\bC$, we write 
$\bM_n$ instead of $\bM_n(\bC)$. Generally, we view $\bM_n(\A)$ as a subalgebra of $B(\H^{(n)})$ in the obvious way (where $\H^{(n)}$ is the $n$-fold direct sum of $\H$), and thus obtain a norm on $\bM_n(\A)$.
If $\E$ is a Hilbert space and $\tau:\A\to B(\E)$ is a linear map, then for every $n\geq 1$ we have an induced linear map $\tau^{(n)}:\bM_n(\A)\to B(\E^{(n)})$ defined as
\[
 \tau^{(n)}([a_{ij}])=[\tau(a_{ij})], \quad [a_{ij}]\in \bM_n(\A).
\]
The map $\tau$ is said to be \emph{completely contractive} or \emph{completely isometric} whenever $\tau^{(n)}$ is contractive or isometric, respectively, for each positive integer $n$. The reader may consult \cite{paulsen2002} for more details on the topic. 

Starting with a unital operator algebra $\A\subset B(\H)$, we can consider the $\rC^*$-algebra $\rC^*(\A)$ that it generates inside of $B(\H)$. However, this $\rC^*$-algebra is typically not invariant under completely isometric isomorphisms of $\A$. In an attempt to construct the smallest possible $\rC^*$-algebra generated by a completely isometric copy of $\A$, Arveson introduced in \cite{arveson1969} non-commutative versions of the Choquet and Shilov boundaries for uniform algebras. We briefly recall some aspects of this construction, and the reader may consult \cite[Chapter 15]{paulsen2002},\cite[Chapter 4]{BLM2004}, or \cite{arveson2008} for details.

Let $\pi:\rC^*(\A)\to B(\H_\pi)$ be a unital $*$-homomorphism. The restriction $\pi|_\A$ is a unital completely contractive map on $\A$, and by virtue of Arveson's extension theorem it thus admits an extension to a unital completely contractive map on $\rC^*(\A)$. Of course, $\pi$ itself is such an extension, but there may well be others in general. Accordingly, we say that $\pi$ has the \emph{unique extension property} with respect to $\A$ if $\pi$ is the unique completely contractive linear extension of $\pi|_\A$ to $\rC^*(\A)$. The following result is found in \cite[Theorem 3.4]{CH2018}  and is based on a result from \cite{richter2010}; it gives a simple concrete criterion that detects the unique extension property. In the statement, any infinite sum is  to be interpreted as being convergent in the strong operator topology.

\begin{theorem}\label{T:sphunitUEP}
 Let $\H$ be a Hilbert space and let $\F\subset B(\H)$ be a countable collection of commuting operators with the property that
 $\sum_{T\in \F}TT^*\leq I.$ Let $\A\subset B(\H)$ denote the unital norm-closed algebra generated by $\F$. Let $\rho:\A\to B(\H_\rho)$ be a unital completely contractive homomorphism with the property that $\rho(T)$ is normal for every $T\in \F$ and
 \[
  \sum_{T\in \F}\rho(T)\rho(T)^*=I.
 \]
Then, there is a unital $*$-homomorphism $\pi:\rC^*(\A)\to B(\H_\rho)$ with the unique extension property with respect to $\A$ such that $\pi|_\A=\rho$.
\end{theorem}

If $\pi$ is an irreducible $*$-representation of $\rC^*(\A)$, then it is said to be a \emph{boundary representation} for $\A$ whenever it has the unique extension property with respect to $\A$. The collection of boundary representation is usually thought of as the \emph{non-commutative Choquet boundary} of $\A$. In identifying boundary representations, the following  general principle is often useful. It is entirely standard, but we record it here for ease of reference throughout the paper.

\begin{lemma}\label{L:compactsplit}
 Let $\A\subset B(\H)$ be a unital operator algebra such that $\rC^*(\A)$ contains the ideal $\fK(\H)$ of compact operators. Let $q:\rC^*(\A)\to \rC^*(\A)/\fK(\H)$ be the quotient map. Let $\pi:\rC^*(\A)\to B(\H_\pi)$ be a unital $*$-homomorphism. Then, there exists a closed subspace $\H_0\subset \H_\pi$ that is reducing for $\pi(\rC^*(\A))$ along with a cardinal number $\kappa$, a unitary operator $U:\H_0\to \H^{(\kappa)}$ and a unital $*$-homomorphism $\sigma:\rC^*(\A)/\fK(\H)\to B(\H_0^\perp)$ such that
\[
 \pi(X)=U^* X^{(\kappa)} U\oplus (\sigma \circ q)(X), \quad X\in \rC^*(\A).
\]
\end{lemma}
\begin{proof}
	See the discussion at the bottom of page 15 along with Corollary 1 on page 20 in \cite{arveson1976inv}. 
 \end{proof}

When every unital $*$-representation of $\rC^*(\A)$ has the unique extension property with respect to $\A$, then we say that $\A$ is \emph{hyperrigid} in $\rC^*(\A)$. This property was introduced in \cite{arveson2011}, where it was conjectured that it holds whenever every irreducible $*$-representation of $\rC^*(\A)$ is a boundary representation for $\A$. This conjecture is still open at the time of this writing.

Regardless of the status of the hyperrigidity conjecture, boundary representations are known to always be plentiful: there always exists a set $\F$ of boundary representations for $\A$ with the property that the map $\bigoplus_{\pi\in \F}\pi$ is completely isometric on $\A$, as shown in \cite{dritschel2005},\cite{arveson2008},\cite{davidson2015}.  

More generally, let $\F$ be a set of unital $*$-representations of $\rC^*(\A)$ that all have the unique extension property with respect to $\A$, and such that the map $\Pi_\F=\bigoplus_{\pi\in \F}\pi$ is completely isometric on $\A$. The set $\F$ can be used to construct a non-commutative version of the Shilov boundary for $\A$. 
We define the \emph{$\rC^*$-envelope} of $\A$ to be
\[
 \rC^*_e(\A)=\Pi_\F(\rC^*(\A)).
\]
Note that $\ker \Pi_\F$ is a closed two-sided ideal of $\rC^*(\A)$ for which \[\rC^*_e(\A)\cong\rC^*(\A)/\ker \Pi_\F\] and that the associated quotient map $\rC^*(\A)\to \rC^*(\A)/\ker \Pi_\F$ is completely isometric on $\A$. (Here and throughout, the symbol $\cong$ denotes the relation of $*$-isomorphism.)
We often identify $\rC^*_e(\A)$ with $\rC^*(\A)/\ker \Pi_\F$.

It can be verified that  $\rC^*_e(\A)$ is invariant under unital completely isometric isomorphisms of $\A$ \cite[Theorem 2.1.2]{arveson1969}. Furthermore, the $\rC^*$-envelope can be interpreted as the essentially unique smallest $\rC^*$-algebra generated by a copy of $\A$, in the sense that it is a quotient of every other such $\rC^*$-algebra \cite[Theorem 4.1]{dritschel2005}.
This object was first proven to exist by Hamana \cite{hamana1979} using different techniques.

When the identity representation of $\rC^*(\A)$ has the unique extension property with respect to $\A$, then the set $\F$ above can be chosen to simply consist of this single representation. So in this case we have $\rC^*_e(\A)\cong \rC^*(\A)$. We thus see that it is desirable to detect when the identity representation enjoys the unique extension property. For this purpose, the best tool is the following result, known as Arveson's \emph{boundary theorem} \cite[Theorem 2.1.1]{arveson1972}.

\begin{theorem}\label{T:bdrythm}
 Let $\A\subset B(\H)$ be a unital operator algebra such that $\rC^*(\A)$ contains the ideal $\fK(\H)$ of compact operators. Let $q:\rC^*(\A)\to \rC^*(\A)/\fK(\H)$ denote the quotient map. Then, the following statements are equivalent.
 \begin{enumerate}[{\rm (i)}]
  \item The identity representation of $\rC^*(\A)$ is a boundary representation for $\A$.
  
  \item The map $q$ is not completely isometric on $\A$.
 \end{enumerate}
\end{theorem}

We now recall some terminology. Let $\fB$ be a unital commutative Banach algebra. We denote the space of its characters by $\Delta(\fB)$. The \emph{Gelfand transform} of $\fB$ is the unital contractive homomorphism
\[
	g:\fB\to C(\Delta(\fB)) 
\]
defined as
\[
 (g(b))(\chi)=\chi(b) \quad \text{for every }b\in \fB, \chi\in \Delta(\fB).
\]
When we wish to emphasize the dependence on the algebra, we sometimes write $g_\fB$ for the Gelfand transform.

The main result of this section shows that whether or not the Gelfand transform of a commutative unital operator algebra is completely isometric can be detected via the $\rC^*$-envelope.

\begin{theorem}\label{T:GTC*env}
  Let $\A\subset B(\H)$ be a norm-closed unital commutative operator algebra. The following statements are equivalent.
  \begin{enumerate}[{\rm (i)}]
   \item The Gelfand transform of $\A$ is completely isometric.
   
   \item The $\rC^*$-envelope of $\A$ is commutative.
   
   \item There is a closed two-sided ideal $\fJ\subset \rC^*(\A)$ such that $\rC^*(\A)/\fJ$ is commutative and the quotient map $q_\fJ:\rC^*(\A)\to \rC^*(\A)/\fJ$ is completely isometric on $\A$.
  \end{enumerate}
\end{theorem}
\begin{proof}
	{\rm (i)} $\Rightarrow$ {\rm (ii)}: Let $g:\A\to C(\Delta(\A))$ denote the Gelfand transform of $\A$, which we assume is completely isometric.  Because the $\rC^*$-envelope is invariant under unital completely isometric isomorphisms of $\A$, we see that $\rC^*_e(\A)\cong\rC^*_e(g(\A))$. But we also know that $\rC^*_e(g(\A))$ is a quotient of $\rC^*(g(\A))\subset C(\Delta(\A))$, so it is commutative.

     {\rm (ii)} $\Rightarrow$ {\rm (i)}:
   Let $\Pi:\A\to \rC^*_e(\A)$ be a unital, completely isometric homomorphism, which exists by construction of the $\rC^*$-envelope. By assumption we have that $\rC^*_e(\A)$ is commutative, so that for $n\geq 1$ and $A\in \bM_n(\A)$ we find
   \begin{align*}
    \|A\|&=\|\Pi^{(n)}(A)\|=\max\{\|\chi^{(n)}(\Pi^{(n)}(A))\|:\chi\in \Delta(\rC^*_e(\A))\}\\
    &=\max\{\|(\chi\circ \Pi)^{(n)}(A)\|:\chi\in \Delta(\rC^*_e(\A))\}\\
    &\leq \max\{\|\chi^{(n)}(A)\|:\chi\in \Delta(\A)\}=\|g^{(n)}(A)\|.
   \end{align*}
 We conclude that the Gelfand transform $g$ is completely isometric. 
 
   {\rm (ii)} $\Rightarrow$ {\rm (iii)}: This follows immediately from the construction of the $\rC^*$-envelope described above.
 
  {\rm (iii)} $\Rightarrow$ {\rm (ii)}: Assume that there is a closed two-sided ideal $\fJ\subset \rC^*(\A)$ such that $\rC^*(\A)/\fJ$ is commutative and the quotient map $q_\fJ:\rC^*(\A)\to \rC^*(\A)/\fJ$ is completely isometric on $\A$.
	It then follows that $\rC^*_e(\A)\cong\rC^*_e(q_\fJ(\A))$ is a quotient of $\rC^*(\A)/\fJ$ and hence it is commutative.
\end{proof}

We extract a consequence that will be useful for our purposes.

\begin{corollary}\label{C:GTC*env}
  Let $\H$ be a Hilbert space with dimension greater than one. Let $\A\subset B(\H)$ be a norm-closed unital commutative operator algebra. Assume that $\rC^*(\A)$ contains the ideal $\fK(\H)$ of compact operators. Consider the following statements.
  \begin{enumerate}[{\rm (i)}]
   \item The Gelfand transform of $\A$ is completely isometric.
   
   \item The identity representation of $\rC^*(\A)$ is not a boundary representation for $\A$. 
  \end{enumerate}
  Then, we have that  {\rm (i)} $\Rightarrow$ {\rm (ii)}. If we assume in addition that $\rC^*(\A)/\fK(\H)$ is commutative, then we have  {\rm (i)} $\Leftrightarrow$ {\rm (ii)}.
\end{corollary}
\begin{proof}
	{\rm (i)} $\Rightarrow$ {\rm (ii)}: We note that $\rC^*(\A)$ is irreducible and not commutative since it contains the compact operators and $\H$ has dimension greater than one. If the identity representation of $\rC^*(\A)$  is a boundary representation for $\A$, then $\rC^*_e(\A)\cong \rC^*(\A)$, which is not commutative. Therefore, the Gelfand transform of $\A$ is not completely isometric by Theorem \ref{T:GTC*env}.
  
	Assume now that  $\rC^*(\A)/\fK(\H)$ is commutative and that the identity representation of $\rC^*(\A)$ is not a boundary representation for $\A$. By Theorem \ref{T:bdrythm}, we see that the quotient map $q:\rC^*(\A)\to \rC^*(\A)/\fK(\H)$ is completely isometric on $\A$. Since $\rC^*(\A)/\fK$ is assumed to be commutative, it follows from Theorem \ref{T:GTC*env} that the Gelfand transform of $\A$ is completely isometric. Thus {\rm (i)} $\Leftrightarrow$ {\rm (ii)} in this case.
\end{proof}


\section{Unitarily invariant kernels and their multiplier algebras}\label{S:kernels}

\subsection{Unitarily invariant kernels with the complete Nevanlinna--Pick property}\label{SS:UENP}

Although the theory of reproducing kernels makes sense in great generality, our focus in this paper will be fairly concrete, so we choose to only recall the relevant definitions in the more specialized setting that we will work in. More details on the general theory can be found in \cite{agler2002} or \cite{PR2016}.

Throughout, the underlying domain will be the open unit ball $\bB_d\subset \bC^d$, for some integer $d\geq 1$. A \emph{kernel} is a function $k:\bB_d\times \bB_d\to \bC$  such that the complex matrix $[k(\lambda_i,\lambda_j)]_{i,j}$ is non-negative definite for any finite subset $\{\lambda_1,\ldots,\lambda_n\}\in \bB_d$. We will always assume that $k$ is a \emph{unitarily invariant} kernel, in the sense that it satisfies the following additional conditions.

\begin{enumerate}[{\rm (a)}]
\item The kernel is \emph{normalized} at $0$, in the sense that
\[
 k(0,w)=1, \quad w\in \bB_d.
\]

 \item The function $k$ is analytic in the first variable and co-analytic in the second.
 \item We have that 
 \[
  k(Uz,Uw)=k(z,w), \quad z,w\in \bB_d
  \]
  for every unitary operator $U:\bC^d\to \bC^d$.
\end{enumerate}
As shown in \cite[Lemma 2.2]{hartz2017isom}, these properties force the kernel to be of the form
\[
 k(z,w)=\sum_{n=0}^\infty a_n \langle z,w \rangle_{\bC^d}^n, \quad z,w\in \bB_d
\]
for some sequence $(a_n)$ of non-negative real numbers such that $a_0=1$. 
Associated to the kernel $k$ is a Hilbert space $\H_k$ of analytic functions on $\bB_d$ which is spanned by the vectors $\{k_w:w\in \bB_d\}$, where
\[
 k_w(z)=k(z,w),\quad z,w\in \bB_d.
\]
Furthermore, these distinguished vectors have the usual reproducing property
\[
 \langle f,k_w \rangle_{\H_k}=f(w)
\]
for every $f\in \H_k$ and every $w\in \bB_d$. 

We say that $k$ is a \emph{regular unitarily invariant kernel} if in addition we have that $a_n>0$ for every $n\geq 0$ and
\[
 \lim_{n\to\infty} \frac{a_n}{a_{n+1}}=1.
\]
As explained in \cite[Section 2]{CH2018}, this last condition insures in particular that the natural domain of definition of the functions in $\H_k$ is $\bB_d$, and not some larger set.

It follows from \cite[Proposition 4.1]{GHX04} that the monomials $\{z^\alpha:\alpha\in \bN_0^d\}$ form an orthogonal basis for $\H$, where
\[
 \bN_0=\{0,1,2,\ldots\}.
\]
Here, we use the standard multi-index notation, so that if $\alpha=(\alpha_1,\ldots,\alpha_d)\in \bN^d_0$ then
\[
 \alpha!=\alpha_1!\alpha_2!\cdots \alpha_d! \qand |\alpha|=\alpha_1+\alpha_2+\cdots+\alpha_d.
\]
Moreover, given a $d$-tuple $r=(r_1,\ldots,r_d)$ of elements in some unital commutative ring we set
\[
 r^\alpha=r_1^{\alpha_1} r_2^{\alpha_2}\cdots r_d^{\alpha_d}.
\]

Among the regular unitarily invariant kernels, we will focus on those that have the \emph{complete Nevanlinna--Pick property}, in the sense that there is a sequence $(b_n)$ of non-negative numbers with the property that
\[
 1-\frac{1}{k(z,w)}=\sum_{n=1}^\infty b_n \langle z,w\rangle^n, \quad z,w\in \bB_d.
\]
Classically, the complete Nevanlinna--Pick property is defined in terms of the solvability of some interpolation problem, but this is equivalent to the previous definition in our context \cite[Lemma 2.3]{hartz2017isom}. A calculation reveals that our standing assumption that $a_0=1$ implies $a_1=b_1>0$. Moreover, it is readily seen that $\sum_{n=1}^\infty b_n\leq 1$.  We will say that the kernel is \emph{maximal} if $\sum_{n=1}^\infty b_n=1$. This is seen to be equivalent to the fact that the series $\sum_{n=1}^\infty a_n$ diverges. The vast majority of our results will require this maximality condition.

Since the kernel $k$ and the Hilbert space $\H_k$ uniquely determine one another, we will often speak of maximal, regular, unitarily invariant complete Nevanlinna--Pick \emph{spaces}, rather than kernels. The following example shows that there are many meaningful instances of such spaces.

\begin{example}\label{E:spaces}
 A standard family of examples of maximal, regular, unitarily invariant kernels on $\bB_d$ with the complete Nevanlinna--Pick property is given by $\{k_s:-1\leq s\}$, where
 \[
  k_s(z,w)=\sum_{n=0}^\infty (n+1)^{s}\langle z,w\rangle^n,\quad z,w\in\bB_d.
 \]
 See \cite[Section 2.4]{CH2018} for details. 
 We note that choosing $s=0$ yields the Drury-Arveson kernel mentioned in the introduction. Moreover, when $d=1$ and $s=-1$, we obtain the famous Dirichlet kernel \cite{EKMR2014}.
 \qed
\end{example}

\subsection{Multiplier algebras and representations}\label{SS:multrep}

Let $\H$ be a regular unitarily invariant complete Nevanlinna--Pick space. Recall that a function $\psi:\bB_d\to\bC$ is a \emph{multiplier} of $\H$ if $\psi f\in \H$ for every $f\in \H$. In this case, the associated multiplication operator $M_\psi:\H\to\H$ is bounded. Throughout, we identify the function $\psi$ with $M_\psi$ whenever convenient. In particular, we may view the algebra $\M(\H)$ of all multipliers of $\H$ as an operator algebra on $\H$, and we define
\[
 \|\psi\|_{\M(\H)}=\|M_\psi\|_{B(\H)}, \quad \psi\in \M(\H).
\]
It is a standard fact that $\M(\H)'=\M(\H)$ in our context, so in particular $\M(\H)$ is weak-$*$ closed. The vectors $\{k_w:w\in \bB_d\}$ span a dense subset of $\H$ and they satisfy
\[
 M_\psi^* k_w=\ol{\psi(w)} k_w
\]
for every $w\in \bB_d$ and every $\psi\in \M(\H)$. 
One consequence of this is that
\begin{equation}\label{Eq:supnorm}
 \|[\psi_{ij}]\|_{\bM_n(\M(\H))}\geq \sup_{z\in \bB_d}\|[\psi_{ij}(z)]\|_{\bM_n}
\end{equation}
for every $n\geq 1$ and every $[\psi_{ij}]\in \bM_n(\M(\H))$.
Another useful consequence is that a bounded net $(\psi_i)_i$ of multipliers converges to $\psi$ in the weak-$*$ topology of $\M(\H)$ if and only if it converges to $\psi$ pointwise on $\bB_d$. Combining \cite[Lemma 2.3 and Proposition 6.4]{hartz2017isom}, we see that the algebra  $\M(\H)$ contains the polynomials. Denoting the standard coordinate functions by $z_1,\ldots,z_d$, we thus have $M_{z_1},\ldots,M_{z_d}\in B(\H)$ and we put 
\[
 M_z=(M_{z_1},\ldots,M_{z_d}).
\]
The following lemma gathers some additional facts that we require throughout.

\begin{lemma}\label{L:projC}
Let $\H$ be a  regular unitarily invariant complete Nevanlinna--Pick space with kernel $k$. Let $(b_n)$ be the sequence of non-negative real numbers such that
\[
1-\frac{1}{k(z,w)}=\sum_{n=1}^\infty b_n \langle z,w \rangle^n, \quad z,w\in \bB_d.
\]
The following statements hold.
\begin{enumerate}[{\rm (i)}]

\item  The row operator 
\[
 \left(b_{|\alpha|}^{1/2}\left(\frac{|\alpha|!}{\alpha!}\right)^{1/2}M_z^\alpha\right)_{|\alpha|\geq 1}
 \]
 is contractive. In fact, the operator 
\[
I-\sum_{n=1}^\infty b_n \sum_{|\alpha|=n}\frac{n!}{\alpha!}M_{z}^\alpha M_z^{\alpha*}
\]
is the orthogonal projection onto $\bC 1$. 

 \item  The function $1-1/k_{\lambda}$ is a contractive multiplier of $\H$ for every $\lambda\in \bB_d$.

\end{enumerate}

\end{lemma}
\begin{proof}

(i) Consider the row operator  $R=\left(b_{|\alpha|}^{1/2}\left(\frac{|\alpha|!}{\alpha!}\right)^{1/2}M_z^\alpha\right)_{|\alpha|\geq 1}$. Then $I-RR^*$ is positive and bounded by \cite[Lemma 5.2]{CH2018}. Thus, the row operator $R$ is indeed contractive. Next, let $\lambda,\mu\in \bB_d$. We compute
\begin{align*}
 \left\langle \left(I-\sum_{n=1}^\infty b_n \sum_{|\alpha|=n}\frac{n!}{\alpha!}M_{z}^\alpha M_z^{\alpha*} \right)k_\lambda,k_\mu \right\rangle &=\left(1-\sum_{n=1}^\infty b_n \sum_{|\alpha|=n} \frac{n!}{\alpha!}\ol{\lambda}^\alpha \mu^\alpha \right)\left\langle k_\lambda,k_\mu \right\rangle\\
 &=\left(1-\sum_{n=1}^\infty b_n \langle \mu,\lambda \rangle^n \right)k(\mu,\lambda)=1.
\end{align*}
On the other hand, if we denote by $E\in B(\H)$ the orthogonal projection onto $\bC 1$, then we find
\[
 \langle Ek_\lambda,k_\mu \rangle=\langle k_\lambda, 1\rangle \langle 1,k_\mu\rangle=1.
\]
Since the vectors $\{k_\lambda:\lambda\in \bB_d\}$ span a dense subset of $\H$, we conclude that
\[
 E=I-\sum_{n=1}^\infty b_n \sum_{|\alpha|=n}\frac{n!}{\alpha!}M_{z}^\alpha M_z^{\alpha*}.
\]

(ii) Consider the row operator
\[
 \Lambda=\left(b_{|\alpha|}^{1/2}\left(\frac{|\alpha|!}{\alpha!}\right)^{1/2}\lambda ^\alpha I\right)_{|\alpha|\geq 1}.
 \]
We find
\begin{align*}
 \Lambda \Lambda^*=\left(\sum_{n=1}^\infty b_n \|\lambda\|^{2n}\right) I\leq \left(\sum_{n=1}^\infty b_n\right) I\leq I.
\end{align*}
Recall now that we proved above that $RR^*\leq I$, so
\[
 R\Lambda^*=\sum_{n=1}^\infty b_n \sum_{|\alpha|=n}\frac{n!}{\alpha!}\ol{\lambda}^\alpha M_z^\alpha
\]
is a contractive multiplier. A straightforward verification shows that
\[
 \Lambda R^* k_\mu=\left(1-\frac{1}{k(\lambda,\mu)} \right)k_\mu
\]
for every $\mu\in \bB_d$, which implies that $ M_{1-1/k_\lambda}=R\Lambda^*$ is a contractive multiplier.

\end{proof}

We let $\A(\H)$ denote the norm-closed subalgebra of $\M(\H)$ generated by the polynomials. Further, we denote by $\fT(\H)$ the C$^*$-subalgebra of $B(\H)$ generated by $\A(\H)$.  We now identify the character spaces of these algebras.

Because of (\ref{Eq:supnorm}), it is readily seen that the functions in $\A(\H)$ extend to be continuous on $\ol{\bB_d}$. In particular, for each $\lambda\in \ol{\bB_d}$ there is a character $\eps_\lambda:\A(\H)\to \bC$ defined as
\[
 \eps_\lambda(\phi)=\phi(\lambda), \quad \phi\in\A(\H).
\]
Conversely, if $\chi:\A(\H)\to\bC$ is a character and if we put
 \[
  \lambda=(\chi(M_{z_1}),\ldots,\chi(M_{z_d})),
 \]
then $\lambda\in \ol{\bB_d}$; this can be seen as in the proof of \cite[Lemma 5.3]{CH2018} using part (i) of Lemma \ref{L:projC}. Since the polynomials are dense in $\A(\H)$, we conclude that $\chi=\eps_\lambda$. Thus, if we let $\Delta(\A(\H))$ denote the character space of $\A(\H)$, then we find
\begin{equation}\label{Eq:charspaceAH}
 \Delta(\A(\H))=\{\eps_\lambda:\lambda\in \ol{\bB_d}\}.
 \end{equation}
We now turn to the character space of $\fT(\H)$. We denote by $\bS_d$ the unit sphere in $\bC^d$, i.e. the topological boundary of $\bB_d$. 

\begin{theorem}\label{T:Toeplitz}
 Let $\H$ be a  regular unitarily invariant complete Nevanlinna--Pick space. Then, the operator $I-\sum_{j=1}^d M_{z_j}M_{z_j}^*$ is compact and the $\rC^*$-algebra $\fT(\H)$ contains the ideal $\fK(\H)$ of compact operators on $\H$. Furthermore, there is a $*$-isomorphism $\Phi:\fT(\H)/\fK(\H)\to C(\bS_d)$ such that
 \[
	 \Phi(M_\phi+\fK(\H))=\phi|_{\bS_d}, \quad \phi\in \A(\H).
 \]
\end{theorem}
\begin{proof}
This follows from \cite[Proposition 4.3 and Theorem 4.6]{GHX04}.
\end{proof}

\subsection{Ideals and quotients}\label{SS:ideals}
In this subsection we introduce the main operator algebras that we study in the paper, and prove some relevant preliminary facts about them. We first point out a very useful correspondence between weak-$*$ closed ideals and invariant subspaces of $\M(\H)$, which can be found in \cite[Theorem 2.2]{CT2019spectrum}.

\begin{theorem}\label{T:lattice}
Let $\H$ be a unitarily invariant complete Nevanlinna--Pick space. Given a weak-$*$ closed ideal $\fa\subset \M(\H)$, the set
\[
 \R(\fa)=\ol{\spn\{M_\phi f:\phi\in \fa, f\in \H\}}
\]
is a closed invariant subspace for $\M(\H)$. Given a closed subspace $\N\subset \H$ that is invariant for $\M(\H)$, the set
\[
 \iota(\N)=\{\psi\in \M(\H): \ran M_\psi\subset \N\}
\]
is a weak-$*$ closed ideal of $\M(\H)$. Moreover, we have that
\[
 \iota(\R(\fa))=\fa, \quad \R(\iota(\N))=\N
\]
for every weak-$*$ closed ideal $\fa\subset \M(\H)$ and every closed subspace $\N\subset \H$ that is invariant for $\M(\H)$.
\end{theorem}

Let $\H$ be a unitarily invariant complete Nevanlinna--Pick space and let $\fa\subset \M(\H)$ be a weak-$*$ closed ideal. Let
\[
 [\fa \H]=\ol{\spn\{M_\phi f:\phi\in \fa, f\in \H\}}
\]
and $\H_\fa=\H\ominus [\fa \H]$. If $\fa$ is a proper ideal of $\M(\H)$, then $\H_\fa$ is necessarily non-zero, as we show next.

\begin{lemma}\label{L:properideal}
 Let $\H$ be a unitarily invariant complete Nevanlinna--Pick space and let $\fa\subset \M(\H)$ be a weak-$*$ closed ideal. The following statements hold.
 \begin{enumerate}[{\rm (i)}]
  \item The ideal $\fa$ is proper if and only if $P_{\H_\fa}1\neq 0$.
  
  \item The ideal $\fa$ is maximal if and only if $\H_\fa$ is one-dimensional. 
 \end{enumerate}
\end{lemma}
\begin{proof}
(i) If $\fa=\M(\H)$, then clearly $\H_\fa=\{0\}$ and $P_{\H_\fa}1=0$. Conversely, if $\fa$ is proper then $1\notin \fa$, whence $1\notin [\fa \H]$ by Theorem \ref{T:lattice}. This is equivalent to $P_{\H_\fa}1\neq 0$.

(ii) If $\fa$ is proper but not maximal, then $\M(\H)/\fa$ has dimension at least two. Hence, there is $\phi\in \fa$ such that the cosets $1+\fa$ and $\phi+\fa$ are linearly independent in $\M(\H)/\fa$. Thus, if $c,d\in \bC$ are not both zero, then $c1+d\phi\notin \fa$, which in turn implies by Theorem \ref{T:lattice} that $c1+d\phi\notin [\fa \H]$. We conclude that $P_{\H_\fa}1$ and $P_{\H_\fa}\phi$ are linearly independent in $\H_\fa$, so that $\H_\fa$ has dimension greater than one. Conversely, if $\fa$ is maximal then we have $\M(\H)=\bC 1+\fa$. Using that $\M(\H)$ is dense in $\H$, we infer that $\bC 1+\fa$ is dense in $\H$ as well. This implies that $\bC P_{\H_\fa} 1=P_{\H_\fa}(\bC 1+\fa)$ is dense in $\H_\fa$, so that $\H_\fa$ is one-dimensional.
\end{proof}

Observe that $\H_\fa$ is a co-invariant subspace for $\M(\H)$.
Let $Z_\fa =(\fz_1,\ldots,\fz_d)$ denote the $d$-tuple of commuting operators on $\H_\fa$ where
\[
	 \fz_j=P_{\H_\fa}M_{z_j}|_{\H_\fa}, \quad 1\leq j\leq d.
\]
Consider now the map $\Gamma:\M(\H)\to B(\H_\fa)$ defined as
\[
 \Gamma(\psi)=P_{\H_\fa}M_{\psi}|_{\H_\fa}, \quad \psi\in \M(\H).
\]
It is readily verified that $\Gamma$ is unital completely contractive homomorphism that is also weak-$*$ continuous. It follows from Theorem \ref{T:lattice} that $\ker \Gamma=\fa$. The map $\Gamma$ should be thought of as the ``$\M(\H)$-functional calculus'' in the sense of \cite{BHM2018}, and thus we use the notation
\[
 \psi(Z_\fa)=\Gamma(\psi), \quad \psi\in \M(\H).
\]
Throughout the paper, we will use the notation $\M_\fa=\Gamma(\M(\H)).$ Thus, 
\[
 \M_\fa=\{\psi(Z_\fa):\psi\in \M(\H)\}.
\]
This algebra will be one of our main objects of study.
It follows from the Commutant Lifting Theorem  \cite[Theorem 5.1]{BTV2001} that $\M_\fa'=\M_\fa$. In particular, $\M_\fa$ is closed in the weak-$*$ topology of $B(\H_\fa)$. Invoking \cite[Theorem 2.3]{CT2019spectrum}, we see that $\Gamma$ induces a unital, completely isometric and weak-$*$ homeomorphic isomorphism $\Gamma_\fa:\M(\H)/\fa\to \M_\fa$ such that
\[
 \Gamma_\fa(M_\psi+\fa)=\psi(Z_\fa), \quad \psi\in \M(\H).
\]
We denote the character space of $\M_\fa$ by $\Delta(\M_\fa)$. Following \cite[Theorem 3.8]{CT2019spectrum}, we define the \emph{support} of $\fa$ to be the set
\[
 \supp \fa=\{(\chi(\fz_1),\ldots,\chi(\fz_d)):\chi\in \Delta(\M_\fa)\}.
\]
If $\M(\H)$ is assumed to satisfy the Corona property, then a more concrete function theoretic description of the support can be given, see \cite[Corollary 3.12]{CT2019spectrum}. In general, it can be shown that $\supp \fa\cap \bB_d$ coincides with the zero set
\[
\{\lambda\in\bB_d:\psi(\lambda)=0 \text{ for every }\psi\in \fa\},
\]
see \cite[Theorem 3.4]{CT2019spectrum} for details. 

Next, we define the second operator algebra on which we focus throughout. We put
\[
 \A_\fa=\ol{\Gamma(\A(\H))}^{\textrm{norm}}
\]
so that
\[
	\A_\fa=\ol{\{\phi(Z_\fa):\phi\in \A(\H)\}}^{\textrm{norm}}.
\]
We also set  $\fT_\fa=\rC^*(\A_\fa)$. We show below that $\fT_\fa$ contains the ideal $\fK(\H_\fa)$ of compact operators, and we denote the corresponding quotient by $\O_\fa=\fT_\fa/\fK(\H_\fa)$. 
The following generalizes \cite[Lemma 7.1]{CH2018}.

\begin{lemma}\label{L:irred}
Let $\H$ be a  regular unitarily invariant complete Nevanlinna--Pick space with kernel $k$. Let $(b_n)$ be the sequence of non-negative real numbers such that
\[
1-\frac{1}{k(z,w)}=\sum_{n=1}^\infty b_n \langle z,w \rangle^n, \quad z,w\in \bB_d.
\]
Let $\fa\subset \M(\H)$ be a proper weak-$*$ closed ideal. Then, the following statements hold.
\begin{enumerate}[{\rm (i)}]
\item The operator  $I-\sum_{j=1}^d \fz_j \fz_j^*$ is compact.

\item We have that 
\[
P_{\H_\fa} P_{\bC 1} |_{\H_\fa} =I-\sum_{n=1}^\infty b_n \sum_{|\alpha|=n} \frac{n!}{\alpha!}Z_\fa^\alpha Z_\fa^{\alpha*}
\] 
and this operator is compact and non-zero. 

\item The $\rC^*$-algebra $\fT_{\fa}$ is irreducible and contains the ideal of compact operators on $\H_{\fa}$. 
\end{enumerate}
\end{lemma}
\begin{proof}
(i) Observe that
\begin{align*}
I-Z_{\fa}Z_{\fa}^*&=P_{\H_\fa}\left(I-\sum_{k=1}^d M_{z_k}M_{z_k}^* \right)|_{\H_\fa}
\end{align*}
which is compact by Theorem \ref{T:Toeplitz}.

(ii) Put 
\[
 Q= I-\sum_{n=1}^\infty b_n \sum_{|\alpha|=n} \frac{n!}{\alpha!}Z_\fa^\alpha Z_\fa^{\alpha*}.
 \]
 Then $Q=P_{\H_\fa}P_{\bC 1}|_{\H_\fa}$ by Lemma \ref{L:projC}(i), so that $Q$ has rank at most one. Moreover, if $Q=0$ then $P_{\H_\fa}P_{\bC 1}=0$, and so $P_{\H_\fa}1=0$, which is impossible by Lemma \ref{L:properideal} since $\fa$ is a proper ideal.

(iii) Let $P\in B(\H_\fa)$ be a self-adjoint projection with $P\neq I$ that commutes with $\A_\fa$, and hence with $\M_\fa$. By the Commutant Lifting Theorem \cite[Theorem 5.1]{BTV2001}, there is $\psi\in \M(\H)$ with $\psi\neq 1$ and $\|\psi\|\leq 1$ such that $P=\psi(Z_\fa)$. By the maximum modulus principle, we see that the sequence $(\psi^n)_n$ converges to $0$ pointwise on $\bB_d$, and thus also in the weak-$*$ topology of $\M(\H)$. In particular, we infer that $(\Gamma(\psi^n))_n$ converges to $0$ in the weak-$*$ topology of $B(\H_\fa)$. Since
\[
P=P^n=\Gamma(\psi^n), \quad n\geq 1
\]
we conclude that $P=0$. This shows that $\A_{\fa}$ is irreducible, and hence so is $\fT_\fa$. Therefore, $\fT_\fa$ is an irreducible $\rC^*$-algebra containing a  non-zero compact operator (by (ii)), and hence it must contain all compact operators on $\H_\fa$ \cite[Corollary 2 page 18]{arveson1976inv}.
\end{proof}

In light of the identication of $\M(\H)/\fa$ with $\M_\fa$, it is natural to wonder if $\A_\fa$ can be identified in an analogous fashion with some quotient of $\A(\H)$. This is possible in some special situations. For instance, \cite[Corollary 8.3]{CH2018} can be used to show that if $\fa$ is a so-called \emph{homogeneous} ideal, then $\A_\fa$ can be identified completely isometrically with $\A(\H)/(\fa\cap \A(\H))$. But this identification does not hold for general ideals. Indeed, even in the classical case of the Hardy space on the unit disc there are non-trivial weak-$*$ closed ideals $\fa$ such that $\fa\cap \A(\H)$ is zero, yet the natural map from $\A(\H)$ onto $\A_\fa$ is not isometric \cite[Corollary 1 page 291]{arveson1972}.

We end this section with an important fact. Recall that a commuting $d$-tuple $U=(U_1,\ldots,U_d)$ of operators is said to be a \emph{spherical unitary} if $U_1,\ldots,U_d$ are all normal and $\sum_{j=1}^d U_j U_j^*=I$.

\begin{lemma}\label{L:sphunitUEPk}
Let $\H$ be a  maximal regular unitarily invariant complete Nevanlinna--Pick space. 
Let $\fa\subset \M(\H)$ be a proper weak-$*$ closed ideal. Let $\rho:\A_\fa\to B(\E)$ be a unital completely contractive representation with the property that
$
(\rho(\fz_1),\ldots,\rho(\fz_d))
$
is a spherical unitary. Then, there is a unital $*$-homomorphism $\pi:\fT_\fa \to B(\E)$ with the unique extension property with respect to $\A_\fa$ such that $\pi|_{\A_\fa}=\rho$.
\end{lemma}
\begin{proof}
The proof follows that of \cite[Proposition 6.1]{CH2018}.
For each multi-index $\alpha$ with $|\alpha|\geq 1$, we let $T_\alpha=b_{|\alpha|}^{1/2}\left( \frac{|\alpha|!}{\alpha!}\right)^{1/2}Z_\fa^\alpha$. Then, we find
\[
\sum_{|\alpha|\geq 1}T_\alpha T_\alpha^*\leq I
\]
by Lemma \ref{L:irred}(ii). On other hand, we see that
\begin{align*}
\sum_{|\alpha|\geq 1}\rho(T_\alpha)\rho(T_\alpha)^*&= \sum_{n=1}^\infty b_n\sum_{|\alpha|=n} \frac{n!}{\alpha !}\rho(Z_\fa)^\alpha \rho(Z_\fa)^{*\alpha}\\
&= \sum_{n=1}^\infty b_n \left(\sum_{j=1}^n \rho(\fz_j)\rho(\fz_j)^* \right)^n=\sum_{n=1}^\infty b_n I=I.
\end{align*}
Note also that each $\rho(T_\alpha)$ is normal by virtue of the Fuglede--Putnam theorem.
Because $\A_\fa$ is the norm-closed unital algebra generated by $\{T_\alpha:|\alpha|\geq 1\}$, we may apply Theorem \ref{T:sphunitUEP} to finish the proof.
\end{proof}

We note here that the maximality condition in the previous theorem cannot be removed; see \cite[Theorem 6.2(b)]{CH2018}.


\section{The connection between Gelfand transforms and boundary representations}\label{S:Gelfbdry}
The main goal of this section is to establish a link between the operator algebras $\A_\fa$ and $\M_\fa$ having a completely isometric Gelfand transform, and the identity representation being a boundary representation for them. As announced in the introduction, we will deal with the two algebras separately.

\subsection{The norm-closed algebra $\A_\fa$}\label{SS:Aa}

We first focus on $\A_\fa$.  For this purpose, we heavily draw from \cite{kennedyshalit2015} and adapt some of the results therein. 
Before proceeding, we introduce some notation. Given a compact subset $K\subset \bC^d$, we denote its \emph{polynomially convex hull} by $\widehat K$, so that $\lambda\in \widehat K$ if and only if
\[
 |p(\lambda)|\leq \max_{w\in K}|p(w)|
\]
for every polynomial $p$. We say that $K$ is \emph{polynomially convex} when $\widehat K=K$. It is well-known that if  $K\subset \ol{\bB_d}$, then $\widehat{K}\cap \bS_d=K\cap \bS_d$; this can readily be verified upon considering the polynomial
\[
 p_\zeta=\frac{1}{2}(1+\langle z,\zeta \rangle)
\]
which peaks at $\zeta\in \bS_d$.

We now identify the character spaces of $\A_\fa$ and $\fT_\fa$, thereby obtaining versions \cite[Proposition 2.1]{kennedyshalit2015} and \cite[Theorem 7.3]{CH2018} applicable to our context. Given a subset $\S\subset \A(\H)$, we write
\[
 \Z_{\ol{\bB_d}}(\S)=\{\lambda\in \ol{\bB_d}:\phi(\lambda)=0 \text{ for every }  \phi\in \S\}.
\]
Moreover, if $\S\subset \M(\H)$ we write
\[
 \Z_{\bB_d}(\S)=\{\lambda\in \bB_d:\psi(\lambda)=0 \text{ for every }  \psi\in \S\}.
\]

\begin{theorem}\label{T:specAaTa}
Let $\H$ be a  regular unitarily invariant complete Nevanlinna--Pick space. Let $\fa\subset \M(\H)$ be a proper weak-$*$ closed ideal. Let $\Delta(\A_\fa)$ and $\Delta(\fT_\fa)$ denote the character spaces of $\A_\fa$ and $\fT_\fa$ respectively.  Then, the following statements hold.
\begin{enumerate}[{\rm(i)}]
 \item For each $\lambda\in \widehat{\supp \fa}$, there is a unique character $\tau_\lambda\in \Delta(\A_\fa)$ such that 
 \[
  \tau_\lambda(\phi(Z_\fa))=\phi(\lambda), \quad \phi\in \A(\H).
 \]
 Moreover, we have that
 \[
  \Delta(\A_\fa)=\{\tau_\lambda:\lambda\in \widehat{\supp \fa}\}.
 \]

\item We have that
\[
 \{(\chi(\fz_1),\ldots,\chi(\fz_d)):\chi\in \Delta(\A_\fa)\}\subset \Z_{\ol{\bB_d}}(\fa\cap \A(\H)).
\]

\item Assume that $\H$ is maximal. Let $\chi\in\Delta(\A_\fa)$ and set $\zeta=(\chi(\fz_1),\ldots,\chi(\fz_d))$.
	If $\zeta\in\bS_d$, then there exists a unique $\pi_\zeta\in\Delta(\fT_\fa)$ such that 
	\[ \pi_\zeta(\phi(Z_\fa))=\chi(\phi(Z_\fa))=\phi(\zeta), \quad \phi\in\A(\H). \]

\item Assume that $\H$ is maximal.   If $\H_\fa$ is one-dimensional, then $\Delta(\A_\fa)=\Delta(\fT_\fa)$. If $\H_\fa$ has dimension greater than one, then 
\[
 \Delta(\fT_\fa)=\{\pi_\zeta:\zeta\in \supp \fa\cap \bS_d\}.
\]
\end{enumerate}
\end{theorem}
\begin{proof}
	(i) By \cite[Theorem I.3.10 (vii)]{muller2007} and by our definition of the support, we see that
\[
 \{(\chi(\fz_1),\ldots,\chi(\fz_d)):\chi\in \Delta(\A_\fa)\}=\widehat{\supp \fa}.
\]
Furthermore, by density of the polynomials in $\A(\H)$, a character $\chi\in \Delta(\A_\fa)$ is uniquely determined by its values on $\fz_1,\ldots,\fz_d$. Both statements follow from this.

(ii) We first make a preliminary observation. Let $w\in \supp \fa$. By (i), we have that $\tau_w\in \Delta(\A_\fa)$. Hence
\[
 \phi(w)=\tau_w(\phi(Z_\fa))=0, \quad \phi\in \fa\cap \A(\H).
\]
This shows that $\supp \fa\subset \Z_{\ol{\bB_d}}(\fa\cap\A(\H))$. Next, let $\chi\in \Delta(\A_\fa)$.
By (i), there is $\lambda\in \widehat{\supp \fa}$ such that $\chi=\tau_\lambda$.
It follows that 
\[
 |\phi(\lambda)|\leq \max_{w\in\supp\fa}|\phi(w)|, \quad \phi\in \A(\H)
 \]
 whence $\lambda\in \Z_{\ol{\bB_d}}(\fa\cap \A(\H))$ by the previous observation. We conclude that
\[
 \{(\chi(\fz_1),\ldots,\chi(\fz_d)):\chi\in \Delta(\A_\fa)\}\subset \Z_{\ol{\bB_d}}(\fa\cap \A(\H)).
\]

 (iii) Note that
	\[
		\sum_{j=1}^d |\chi(\fz_j)|^2=\|\zeta\|^2_{\bC^d}=1,
	\]
	so we may invoke Lemma \ref{L:sphunitUEPk} to see that there is a unique $\pi_\zeta\in\Delta(\fT_\fa)$ such that 
	\[ \pi_\zeta(\phi(Z_\fa))=\chi(\phi(Z_\fa))=\phi(\zeta), \quad \phi\in\A(\H). \]

(iv) If $\H_\fa$ is one-dimensional, then clearly $\A_\fa=\fT_\fa$. Suppose therefore that $\H_\fa$ has dimension greater than one. We wish to show that 
\[
 \Delta(\fT_\fa)=\{\pi_\zeta:\zeta\in \supp \fa\cap \bS_d\}.
\]
One inclusion follows from (iii). To see the reverse inclusion, let $\chi\in \Delta(\fT_\fa)$ and put 
\[
\zeta=(\chi(\fz_1),\ldots,\chi(\fz_d)).
\]
Then, $\chi|_{\A_\fa}\in \Delta(\A_\fa)$ so that $\zeta\in \widehat{\supp\fa}$ and $\chi|_{\A_\fa}=\tau_\zeta$ by (i). 
By Lemma \ref{L:irred}, we see that $\fT_\fa$ contains the ideal $\fK(\H_\fa)$ of compact operators and that $I-\sum_{j=1}^d \fz_j \fz_j^*$ is compact. Since $\H_\fa$ has dimension greater than one, it follows from Lemma \ref{L:compactsplit} that $\fK(\H_\fa)\subset \ker \chi$, whence
\[
\|\zeta\|_{\bC^d}^2= \sum_{j=1}^d |\chi(\fz_j)|^2=1
\]
and $\zeta\in \bS_d$. Thus
\[
 \zeta\in \widehat{\supp \fa}\cap \bS_d=\supp \fa\cap \bS_d
\]
in light of the remark preceding the theorem. Finally, the equality $\chi|_{\A_\fa}=\tau_\zeta$ implies $\chi=\pi_\zeta$ since $\fT_\fa=\rC^*(\A_\fa)$.

\end{proof}

We now aim to calculate the essential norm of an element of $\A_\fa$, using the same argument as in \cite[Theorem 3.3]{kennedyshalit2015}. 
To this end, recall that $\fT_\fa$ always contains the ideal $\fK(\H_\fa)$ of compact operators by Lemma \ref{L:irred}, and that we write $\O_\fa=\fT_\fa/\fK(\H_\fa)$. We use the symbols $g_{\fB}$ to denote the Gelfand transform of a commutative unital Banach algebra $\fB$.

\begin{theorem}\label{T:Aaessnorm}
	Let $\H$ be a  maximal regular unitarily invariant complete Nevanlinna--Pick space. Let $\fa\subset \M(\H)$ be a proper weak-$*$ closed ideal and let $q:\fT_\fa\to \O_\fa$ denote the quotient map. Let $n\geq 1$ and let $A=[\phi_{ij}(Z_\fa)]_{ij}\in \bM_n(\A_\fa)$.  Then, we have that
	\begin{equation}\label{Eq:QltTltA}
		\|q^{(n)}(A)\|\leq \|g^{(n)}_{\fT_\fa}(A)\| \leq \|g^{(n)}_{\A_\fa}(A)\|.
	\end{equation}
	If we assume in addition that $\H_\fa$ has dimension greater than one, then we have
	\[
		\|q^{(n)}(A)\|= \|g_{\fT_\fa}^{(n)}(A)\| = \|g_{\O_\fa}^{(n)}(q^{(n)}(A))\| = \|g_{q(\A_\fa)}^{(n)}(q^{(n)}(A))\| .
	\]
	In particular, the Gelfand transform of $q(\A_\fa)$ is completely isometric in this case.
\end{theorem}
\begin{proof}
	In what follows, when $\chi$ is a character of an algebra containing $\fz_1,\ldots,\fz_d$, we put $\chi(Z_\fa)=(\chi(\fz_1),\ldots,\chi(\fz_d))$, with similar natural abbreviations being used where appropriate.
	
	We may assume that $\O_\fa\subset B(\E)$ for some Hilbert space $\E$.
	It follows from Lemma \ref{L:irred}(i) that
	\[
 		I=\sum_{j=1}^d q(\fz_j) q(\fz_j)^*.
	\]
	By \cite[Proposition 2]{athavale1990}, we see that there is a Hilbert space $\E'$ containing $\E$ along with a spherical unitary $U=(U_1,\ldots,U_d)$ acting on $\E'$ and leaving $\E$ co-invariant, with the property that $U^*_j|_{\E}=q(\fz_j)^*$ for every $1\leq j\leq d$. 
	Requiring that $U^*$ be minimal among the normal dilations of $q(Z_{\fa})^*$, we can arrange that
	\[
		\sigma_{\Ta}(U^*)\subset \sigma_{\Ta}(q(Z_\fa)^*)
	\]
	by virtue of \cite{putinar1984}. In turn, as explained in \cite[Section 3]{richter2010} this implies
	\[
		\sigma_{\Ta}(U)\subset \sigma_{\Ta}(q(Z_\fa)).
	\]
	On the other hand, we always have that
	\[
		\sigma_{\Ta}(q(Z_\fa))\subset \{\chi(q(Z_\fa)):\chi\in \Delta(q(\A_\fa))\}
	\]
	by \cite[Proposition IV.25.3]{muller2007}.
	Now, given $\chi\in \Delta(q(\A_\fa))$, we see that $\chi\circ q\in \Delta(\A_\fa)$ and so
	\[
		\sigma_{\Ta}(q(Z_\fa))\subset \{\chi(Z_\fa):\chi\in \Delta(\A_\fa)\}.
	\]
	In particular, we obtain
	\[
		\sigma_{\Ta}(q(Z_\fa)) \cap \bS_d \subset \{\chi(Z_\fa):\chi\in \Delta(\A_\fa)\}\cap \bS_d \subset \{\chi(Z_\fa):\chi\in \Delta(\fT_\fa)\}
	\]
	where the last equality follows Theorem \ref{T:specAaTa}(iii).
	Since
	\[
		\sigma_{\Ta}(U)=\{\chi(U):\chi\in \Delta(\rC^*(U))\}\subset \bS_d
	\]
	by \cite[Proposition 7.2]{curto1988}, we find
	\begin{equation}\label{Eq:TSpecU}
		\sigma_{\Ta}(U)\subset\sigma_{\Ta}(q(Z_\fa)) \cap \bS_d \subset  \{\chi(Z_\fa):\chi\in \Delta(\fT_\fa)\} \subset \{\chi(Z_\fa):\chi\in\Delta(\A_\fa)\}.
	\end{equation}
	Observe now that
	\[
		q^{(n)}(A) =[q(\phi_{ij}(Z_\fa))]_{ij}=[(\phi_{ij}(q(Z_\fa))]_{ij}=P_{\E^{(n)}}[\phi_{ij}(U)]_{ij}|_{\E^{(n)}}
	\]
	so we find
	\[
		 \| q^{(n)}(A)\|\leq \|[\phi_{ij}(U)]_{ij}\|.
	\]
	Next, we know that $\bM_n(\rC^*(U))$ is $*$-isomorphic to $C(\Delta(\rC^*(U)),\bM_n)$ and
	\begin{align*}
		\|[\phi_{ij}(U)]_{ij}\|_{\bM_n(\rC^*(U))} &= \max\{ \|[\phi_{ij}(\chi(U))]_{ij}\|_{\bM_n}:\chi\in \Delta(\rC^*(U)) \} \\
		&=\max\{ \|[\phi_{ij}(\zeta)]_{ij}\|_{\bM_n}:\zeta\in\sigma_{\Ta}(U) \}.
	\end{align*}
	Combining these facts with \eqref{Eq:TSpecU}, we have
	\begin{align*}
		\|q^{(n)}(A)\| &\leq  \max\{\|[\phi_{ij}(\chi(Z_\fa))]_{ij}\|:\chi\in \Delta(\fT_\fa)\} \\
									 &\leq \max \{\|[\phi_{ij}(\chi(Z_\fa))]_{ij}\|:\chi\in\Delta(\A_\fa)\},
	\end{align*}
	which is \eqref{Eq:QltTltA}.

	If $\H_\fa$ has dimension greater than one, then every character of $\fT_\fa$ must annihilate $\fK(\H_\fa)$ by Lemma \ref{L:compactsplit}. Thus, given $\chi\in \Delta(\fT_\fa)$, we can find a character $\widehat\chi\in \Delta(\O_\fa)$ such that $\chi=\widehat\chi\circ q$.
	This implies that
	\[
		\|g^{(n)}_{\fT_\fa}(A)\| = \|(g_{\O_\fa}\of q)^{(n)}(A)\| \leq \|(g_{q(\A_\fa)}\of q)^{(n)}(A)\| \leq \|q^{(n)}(A)\|. 
	\]
	Applying \eqref{Eq:QltTltA}, we conclude that
	\[ 
		\|q^{(n)}(A)\|=\|g^{(n)}_{\fT_\fa}(A)\| = \|(g_{\O_\fa}\of q)^{(n)}(A)\| = \|(g_{q(\A_\fa)}\of q)^{(n)}(A)\|
	\]
	in this case.
\end{proof}

We note that the second (stronger) assertion of the previous theorem fails when $\H_\fa$ is one-dimensional. Indeed, in that case we have $q(I)=0$ while the identity map is a character of $\fT_\fa$. Furthermore, we mention that the method of proof above does not typically adapt to apply to $\M_\fa$, in view of the following example.

\begin{example}\label{E:cpactquotientFX}
Let $d\geq 2$ and let $\H$ be the Drury--Arveson space on $\bB_d$. Let $\fa=\{0\}$ so that $\H=\H_\fa$ and $\M_\fa=\M(\H)$. Let $q:\rC^*(\M(\H))\to \rC^*(\M(\H))/\fK(\H)$ be the quotient map. By \cite[Theorem 1.1]{FX2011}, there is $\psi\in \M(\H)$ such that
\[
 \|q(M_\psi)\|>\sup_{z\in \bB_d}|\psi(z)|.
\]
Next, let $\chi:q(\M(\H))\to \bC$ be a character.
Then, $\chi\circ q|_{\M(\H)}$ is a character of $\M(\H)$, so by the Corona theorem \cite{costea2011}, we see that 
\[
 |\chi(q(M_\psi))|\leq \sup_{z\in \bB_d}|\psi(z)|.
\]
We conclude that 
\[
	\|q(M_\psi)\|>\max\{ |\chi(q(M_\psi))|:\chi\in \Delta(q(\M(\H))\}.
\]
\qed
\end{example}

Before we give some consequences of Theorem \ref{T:Aaessnorm}, we require a maximum modulus principle for $\M_\fa$ that is of independent interest.

\begin{theorem}\label{T:maximummodulus}
Let $\H$ be a  regular unitarily invariant complete Nevanlinna--Pick space. Let $\fa\subset \M(\H)$ be a proper weak-$*$ closed ideal and let $\lambda\in \Z_{\bB_d}(\fa)$. Let $M=[\psi_{ij}(Z_\fa)]_{ij}\in \bM_n(\M_\fa)$ be such that $\|M\|=\| [\psi_{ij}(\lambda)]_{ij}\|_{\bM_n}$. Then, there are unit vectors $v=(v_1,\ldots,v_n)$ and $w=(w_1,\ldots,w_n)$ in $\bC^n$ with the property that
 \[
 \sum_{j=1}^n v_j\psi_{ij}(Z_\fa)   =\|M\| w_i I_{\H_\fa}, \quad 1\leq i\leq n.
\]
\end{theorem}
\begin{proof}
Put $M(\lambda)=[\psi_{ij}(\lambda)]_{ij}$. By assumption, we may choose unit vectors $v,w\in \bC^n$ with the property that
\[
 \langle  M(\lambda)v,w\rangle_{\bC^n}=\|M\|.
\]
Because $\lambda\in \Z_{\bB_d}(\fa)$, we have $k_\lambda\in \H_\fa$.
Define unit vectors $\xi,\eta\in\H_\fa^{(n)}$ by
\[
 \xi=\left( v_1 \frac{k_\lambda}{\|k_\lambda\|}, \ldots,  v_n \frac{k_\lambda}{\|k_\lambda\|}\right), \quad  \eta=\left( w_1 \frac{k_\lambda}{\|k_\lambda\|}, \ldots,  w_n \frac{k_\lambda}{\|k_\lambda\|}\right).
\]
We compute
\begin{align*}
 \langle M\xi,\eta \rangle_{\H_\fa^{(n)}}&=\sum_{i,j=1}^n \frac{1}{\|k_\lambda\|^2}v_j \ol{w_i} \langle \psi_{ij}(Z_\fa)k_\lambda,k_\lambda\rangle_{\H_\fa}\\
 &=\sum_{i,j=1}^n v_j \ol{w_i}\psi_{ij}(\lambda)\\
 &=\langle M(\lambda)v,w\rangle_{\bC^n}=\|M\|.
\end{align*}
By the Cauchy-Schwarz inequality, we infer that $M\xi=\|M\| \eta$.
In other words, for each $1\leq i\leq n$ we find
\[
 \sum_{j=1}^n \psi_{ij}(Z_\fa) v_j k_\lambda=\|M\| w_i k_\lambda
\]
which is equivalent to
\[
 \left(\sum_{j=1}^n v_j \psi_{ij}-\|M\| w_i\right)k_\lambda\in [\fa\H].
\]
From Lemma \ref{L:projC}(ii), we know that $1/k_\lambda$ is a multiplier so in fact
\[
 \sum_{j=1}^n v_j \psi_{ij}-\|M\| w_i\in [\fa \H]
\]
for every $1\leq i\leq n$. In turn, by Theorem \ref{T:lattice}, this implies that
\[
 \sum_{j=1}^n v_j \psi_{ij}-\|M\| w_i\in \fa
\]
 whence
\[
 \sum_{j=1}^n  \psi_{ij}(Z_\fa)v_j=\|M\| w_i I_{\H_\fa}
\]
for every $1\leq i\leq n$. 
\end{proof}

We now illustrate the importance in the previous theorem that $\lambda$ be an interior zero of the ideal;  the conclusion fails if $\lambda$ is simply an interior point of the ball.

\begin{example}\label{E:maxprinciple}
Let $\H$ be the Hardy space on the unit disc. Let $\theta\in \M(\H)$ be an inner function such that if $\fa=\theta \M(\H)$ then there is $\psi\in \M(\H)$ with $\psi\notin \fa+\bC 1$. Choose  distinct points $w_1,w_2\in\bD$ such that $\psi(w_1),\psi(w_2),\theta(w_1),\theta(w_2)$ are all non-zero. Moreover, choose $c_1,c_2\in \bC$ such that
\[
 |\psi(w_1)+c_1\theta(w_1)|>\sup_{z\in \bD}|\psi(z)|=\|M_\psi\|
\]
and
\[
 \psi(w_2)+c_2\theta(w_2)=0.
\]
Next, let $\phi\in \A(\H)$ be such that $\phi(w_1)=c_1,\phi(w_2)=c_2$. We find
\[
 |(\psi+\theta\phi)(w_1)|>\|M_\psi\|\geq \|\psi(Z_\fa)\|
\]
and $(\psi+\theta\phi)(w_2)=0$. Let $r=\max\{|w_1|,|w_2|\}$, so that $0<r<1$. Since the function $|\psi+\theta\phi|$ is continuous on the closed disc of radius $r$ centred at the origin, we conclude that there is $w\in \bD$ such that
\[
 |(\psi+\theta\phi)(w)|=\|\psi(Z_\fa)\|.
\]
Nevertheless, we see that $\psi(Z_\fa)$ is not a constant multiple of the identity, since $\psi\notin \fa+\bC 1$.
 \qed
\end{example}

As a first application of Theorem \ref{T:Aaessnorm}, we compare the norm of the Gelfand transform of an element of $\A_\fa$ to its essential norm in some extremal cases.

\begin{corollary}\label{C:uniformquotient}
	Let $\H$ be a  maximal regular unitarily invariant complete Nevanlinna--Pick space.
	Let $\fa\subset \M(\H)$ be a proper weak-$*$ closed ideal with the property that $\H_\fa$ is infinite-dimensional.	Let $q:\fT_\fa\to \O_\fa$ be the quotient map.
	Let $A=[\phi_{ij}(Z_\fa)]_{ij}\in \bM_n(\A(\H))$. Then, we have that
	\[
 	\|A\|=\max\{\|\chi^{(n)}(A) \|_{\bM_n}:\chi\in \Delta(\A_\fa) \}
	\]
	if and only if $\|A\|=\|q^{(n)}(A)\|$.
\end{corollary}
\begin{proof}
If $\|A\|=\|q^{(n)}(A)\|$, then Theorem \ref{T:Aaessnorm} implies that 
\[
 \|A\|=\max\{\|\chi^{(n)}(A) \|_{\bM_n}:\chi\in \Delta(\A_\fa) \}.
\]
Conversely, assume that the previous equality holds. By Theorem \ref{T:specAaTa}(i), we see that
\[
 \|A\|=\max\{\|[\phi_{ij}(\lambda)]\|_{\bM_n}:\lambda\in \widehat{\supp \fa} \}.
\]
Let $\lambda\in  \widehat{\supp \fa}$ and choose unit vectors $v,w\in \bC^n$ such that
\[
 \|[\phi_{ij}(\lambda)]\|_{\bM_n}=|\langle [\phi_{ij}(\lambda)]v,w \rangle|.
\]
Since the polynomials are dense in $\A(\H)$, by definition of the polynomial convex hull we conclude that
\[
 \|[\phi_{ij}(\lambda)]\|_{\bM_n}\leq \max_{\mu\in \supp \fa}|\langle [\phi_{ij}(\mu)]v,w \rangle|\leq \max_{\mu\in \supp\fa}\|[\phi_{ij}(\mu)]\|_{\bM_n}.
\]
This shows that
\[
 \|A\|=\max\{\|[\phi_{ij}(\lambda)]\|_{\bM_n}:\lambda\in \supp \fa \}.
\]
Choose $\lambda\in \supp \fa$ such that 
\[ \|A\|=\|[\phi_{ij}(\lambda)]_{ij}\|_{\bM_n}. \]
Assume first that $\lambda\in \supp\fa\cap \bS_d$.
By Theorem \ref{T:specAaTa}(iii), there exists  $\pi\in\Delta(\fT_\fa)$ such that $\pi(Z_\fa)=\lambda$ and
\[
	\|A\|=\|[\phi_{ij}(\lambda)]_{ij}\|_{\bM_n} = \|\pi^{(n)}(A)\|_{\bM_n} \leq \|g_{\fT_\fa}^{(n)}(A)\| \leq \|A\|.
\]
That is, $\|A\|=\|g_{\fT_\fa}^{(n)}(A)\|$.
Invoking Theorem \ref{T:Aaessnorm}, we conclude that $\|A\|=\|q^{(n)}(A)\|$ as desired.
Thus, it remains only to deal with the case where $\lambda\in  \supp \fa\cap \bB_d=\Z_{\bB_d}(\fa)$. In this case, it follows from Theorem \ref{T:maximummodulus} that there are unit vectors $v,w\in \bC^n$ with the property that
\[
\sum_{j=1}^n v_j \phi_{ij}(Z_\fa)= \|A\| w_i I_{\H_\fa}, \quad 1\leq i\leq n.
\]
Therefore, for each $\chi\in \Delta(q(\A_\fa))$ we find
\[
 \sum_{j=1}^n v_j (\chi\circ q)(\phi_{ij}(Z_\fa))= \|A\| w_i , \quad 1\leq i\leq n
\]
or
\[
 \langle (\chi\circ q)^{(n)}(A)v,w\rangle=\|A\|.
\]
Since $v$ and $w$ are unit vectors, this forces $\| (\chi\circ q)^{(n)}(A)\|_{\bM_n}=\|A\|$ whenever $\chi\in \Delta(q(\A_\fa))$.
Finally, because $\H_\fa$ is infinite-dimensional, the algebra $q(\A_\fa)$ is non-zero and it has at least one character, so that we necessarily have $\|A\|=\|q^{(n)}(A)\|.$
\end{proof}

It seems relevant given the previous result to mention that we generally do not have that
\[
 \|q(b)\|=\max\{|\chi(b)|:\chi\in \Delta(\A_\fa) \}
\]
for every $b\in \A_\fa$. Indeed, counter-examples can easily be constructed even on the classical Hardy space on the disc using ideals with support equal to $\{0,1\}$.

Our second application of Theorem \ref{T:Aaessnorm}  is one of our main results. It is a generalization of part of \cite[Proposition 4.10 and Theorem 4.12]{kennedyshalit2015} and uses essentially the same scheme of proof.

\begin{corollary}\label{C:idbdryessnorm}
Let $\H$ be a maximal regular unitarily invariant complete Nevanlinna--Pick space. Let $\fa\subset \M(\H)$ be a proper weak-$*$ closed ideal. The following statements are equivalent.
\begin{enumerate}[{\rm (i)}]
 \item The algebra $\A_\fa$ is hyperrigid in $\fT_\fa$.
 
 \item The $d$-tuple $Z_\fa$ is essentially normal and the identity representation of $\fT_\fa$ is a boundary representation for $\A_{\fa}$.
\end{enumerate}
\end{corollary}
\begin{proof}
The statement is trivial when $\H_\fa$ is one-dimensional, so we assume without loss of generality that $\H_\fa$ has dimension greater than one. Let $q:\fT_\fa\to\O_\fa$ be the quotient map.

(i) $\Rightarrow$ (ii): Assume that $\A_\fa$ is hyperrigid in $\fT_\fa$. In particular,  the identity representation of $\fT_\fa$ is a boundary representation for $\A_{\fa}$; we recall that it is irreducible by Lemma \ref{L:irred}(iii). Moreover, $q(\A_\fa)$ is hyperrigid in $\O_\fa$ by \cite[Corollary 2.2]{arveson2011}. This means that the identity representation of $\O_\fa$ has the unique extension property with respect to $q(\A_\fa)$, whence $\rC^*_e(q(\A_\fa))\cong \O_\fa$ as explained in Section \ref{S:OA}. On the other hand, it follows from Theorem \ref{T:Aaessnorm} that the Gelfand transform of $q(\A_\fa)$ is completely isometric. We conclude that $\rC^*_e(q(\A_\fa))$ is  commutative by Theorem \ref{T:GTC*env}, and thus so is $\O_\fa$. This says precisely that $Z_\fa$ is essentially normal. 

 (ii) $\Rightarrow$ (i): Assume that the $d$-tuple $Z_\fa$ is essentially normal and the identity representation of $\fT_\fa$ is a boundary representation for $\A_{\fa}$. Now, the unique extension property is preserved upon taking direct sums \cite[Proposition 4.4]{arveson2011}. Thus, by virtue of Lemma \ref{L:compactsplit}, to establish the hyperrigidity of $\A_\fa$ it suffices to check that if $\sigma$ is a unital $*$-representation of $\O_\fa$, then $\sigma\circ q$  has the unique extension property with respect to $\A_\fa$. To see this, we notice that $I=\sum_{j=1}^d q(\fz_j) q(\fz_j)^*$ by Lemma \ref{L:irred}(i). Therefore, we must also have that 
\[
I=\sum_{j=1}^d (\sigma \circ q)(\fz_j) (\sigma \circ q)(\fz_j)^*.
\]
Moreover, because $Z_\fa$ is essentially normal by assumption,  $(\sigma\circ q)(\fz_j)$ must be normal for each $1\leq j\leq d$. The fact that $\sigma \circ q$ has the unique extension property with respect to $\A_\fa$ now follows immediately from Lemma \ref{L:sphunitUEPk}.
\end{proof}

In light of the previous theorem, we see that it is of great interest to be able to  determine when the  identity representation of $\fT_\fa$ is a boundary representation for $\A_{\fa}$. The following is one of our main results, and it addresses this question.

\begin{theorem}\label{T:unifquotientAa}
Let $\H$ be a maximal regular unitarily invariant complete Nevanlinna--Pick space. Let $\fa\subset \M(\H)$ be a proper weak-$*$ closed ideal. If $\H_\fa$ has dimension greater than one, then the following statements are equivalent.
\begin{enumerate}[{\rm(i)}]

\item The Gelfand transform of $\A_\fa$ is completely isometric.

 \item The $\rC^*$-envelope of $\A_\fa$ is commutative.

\item The identity representation of $\fT_\fa$ is not a boundary representation for $\A_\fa$.
\end{enumerate}
If $\H_\fa$ is one-dimensional, then $\A_\fa=\fT_\fa\cong \bC$ and every irreducible $*$-representation of $\fT_\fa$ is a boundary representation for $\A_\fa$.
\end{theorem}
\begin{proof}
When $\H_\fa$ is one-dimensional, the statement is trivial. We suppose henceforth that $\H_\fa$ has dimension greater than one. 

(i) $\Leftrightarrow$ (ii): This is immediate from Theorem \ref{T:GTC*env}.

(i) $\Rightarrow$ (iii): Since $\H_\fa$ has dimension greater than one and $\fT_\fa$ contains the compact operators by Lemma \ref{L:irred}(iii), this follows immediately from Corollary \ref{C:GTC*env}.

(iii) $\Rightarrow$ (i): It follows from Theorem \ref{T:bdrythm} that the quotient map $q:\fT_\fa\to \O_\fa$ is completely isometric on $\A_\fa$. Theorem \ref{T:Aaessnorm} then implies that the Gelfand transform of $\A_\fa$ is completely isometric.
\end{proof}

We now look at some concrete examples. First, we address the classical situation of the Hardy space on the unit disc, which was completely elucidated by Arveson in \cite{arveson1969,arveson1972}.

\begin{example}\label{E:HardyAaGelf}
 Let $\H=H^2_1$ be the Hardy space on the unit disc. In this special case, the algebras $\A(\H)$ and $\M(\H)$ can be identified completely isometrically with classical algebras of functions. We briefly hint at the well-known details.
 
 Recall that $H^\infty(\bD)$ is the algebra of all bounded holomorphic functions on $\bD$. It is a Banach algebra when equipped with the norm
 \[
  \|\psi\|=\sup_{\lambda\in \bD}|\psi(\lambda)|, \quad \psi\in H^\infty(\bD).
 \]
The disc algebra $A(\bD)$ is the closed subalgebra of $H^\infty(\bD)$ consisting of those continuous functions on $\ol{\bD}$ that are holomorphic on $\bD$.  By Inequality (\ref{Eq:supnorm}), we see that there is a unital, injective, contractive homomorphism $\Theta:\M(\H)\to H^\infty(\bD)$ such that
\[
 \Theta(M_\psi)=\psi, \quad \psi\in \M(\H).
\]
In fact, it is known that this embedding is isometric and surjective. As we show next, more is true.

  Let $\bT\subset \bC$ denote the unit circle, let $m$ denote Lebesgue measure on $\bT$ and let $\E=L^2(\bT,m)$. Identifying a function in $H^\infty(\bD)$ with its $[m]$-almost everywhere defined boundary values on $\bT$, we see that $A(\bD)\subset H^\infty(\bD)\subset L^\infty(\bT,m)$, so that in particular $A(\bD)$ and $H^\infty(\bD)$  are operator algebras. 
Up to unitary equivalence we have $\H\subset \E$ and
 \[
  M_\psi=\psi(U)|_\H, \quad \psi\in \M(\H)
 \]
where $U\in B(\E)$ denotes the unitary operator of multiplication by the coordinate $\zeta$. In particular, given $[\psi_{ij}]_{ij}\in \bM_n(\M(\H))$, we see that
\begin{align*}
 \|[M_{\psi_{ij}}]_{ij}\|_{\bM_n(\M(\H))}
 &\leq \|[\psi_{ij}]_{ij}\|_{\bM_n(L^\infty(\bT,m))}\\
 &=\sup_{\lambda\in \bD} \|[\psi_{ij}(\lambda)]_{ij}\|_{\bM_n}.
\end{align*}
The converse inequality always holds, by Inequality (\ref{Eq:supnorm}). Thus, the map $\Theta$ is a unital completely isometric isomorphism and $\Theta(\A(\H))=A(\bD)$. 
Now, $A(\bD)\subset C(\ol{\bD})$ and $H^\infty(\bD)\subset L^\infty(\bT,m)$,  and $C(\ol{\bD})$ and $L^\infty(\bT,m)$ are unital commutative $\rC^*$-algebras. It is then well known  that the Gelfand transforms of $\A(\H)$ and $\M(\H)$ must be isometric, and hence completely isometric by \cite[Theorem 3.9]{paulsen2002}. This corresponds to the situation where $\fa=\{0\}$.

We now consider non-trivial quotients. It follows from Beurling's theorem and Theorem \ref{T:lattice} that every non-zero weak-$*$ closed ideal of $\M(\H)$ is a principal ideal generated by an inner function.  Let $\theta\in \M(\H)$ be such an inner function and let $\fa=\theta\M(\H)$. Assume that $\H_\fa$ has dimension greater than one.  Let $K\subset \ol{\bD}$ be the closure of the zero set of $\theta$, along with the points on the unit circle $\bT$ contained in an arc across which $\theta$ cannot be continued holomorphically. Then, $K=\supp \fa$ in this case \cite[Theorem 2.4.11]{bercovici1988}. It follows from \cite[Corollary 1 page 291]{arveson1972} that the identity representation of $\fT_\fa$ is a boundary representation for $\A_\fa$ if and only if $K\cap \bT$ is a proper subset of the circle. By Theorem \ref{T:unifquotientAa}, we conclude that the Gelfand transform of $\A_\fa$ is completely isometric if and only if $K\cap \bT=\bT$.
\qed
\end{example}

For the Drury--Arveson space in several variables, the situation is more complicated. 

\begin{example}\label{E:DAAaGelf}
Let $\H=H^2_2$ be the Drury--Arveson space on $\bB_2$. Let $\fa=\{0\}$. Then, $\H_\fa=\H$ and $\A_\fa=\A(\H)$. It is known (see \cite[Theorem 3.3]{arveson1998}) that the quotient map $q:\fT_\fa\to \O_\fa$ is not completely isometric on $\A_\fa$ in this case, whence the identity representation of $\fT_\fa$ is a boundary representation for $\A_\fa$ by Theorem \ref{T:bdrythm}. By Theorem \ref{T:unifquotientAa}, we conclude that the Gelfand transform of $\A_\fa$ is not completely isometric.

Next, let $\fb\subset \M(\H)$ be the weak-$*$ closed ideal generated by $z_2$.  In this case, $\A_\fb$ is unitarily equivalent to $\A(H^2_1)$, which has a completely isometric Gelfand transform as seen in Example \ref{E:HardyAaGelf}.
\qed
\end{example}

In Section \ref{S:polyapprox}, we will explore further concrete conditions on $\fa$ that guarantee that the Gelfand transform of $\A_\fa$ is \emph{not} completely isometric. In light of Corollary \ref{C:idbdryessnorm} and Theorem \ref{T:unifquotientAa}, this will complement \cite[Theorem 4.12]{kennedyshalit2015} (see \cite{kennedyshalit2015corr} for updated details).

\subsection{The full algebra $\M_\fa$}\label{SS:Ma}

Up to now, we devoted our attention to the norm-closed unital algebra $\A_\fa$ generated by $\fz_1,\ldots,\fz_d$ inside of $\M_\fa$. We now turn our focus to the ``full'' algebra $\M_\fa$, and examine when its Gelfand transform is completely isometric. The first step we take in our analysis of $\M_\fa$ is the following, which provides an analogous tool to Theorem \ref{T:unifquotientAa}. Note that $\rC^*(\M_\fa)$ trivially contains $\fT_\fa=\rC^*(\A_\fa)$, and in particular it contains the ideal of compact operators on $\H_\fa$ by Lemma \ref{L:irred}(iii).

\begin{corollary}\label{C:Maunifalg}
 Let $\H$ be a  regular unitarily invariant complete Nevanlinna--Pick space. Let $\fa\subset \M(\H)$ be a proper weak-$*$ closed ideal such that $\H_\fa$ has dimension greater than one. Consider the following statements.
   \begin{enumerate}[{\rm (i)}]
   \item The Gelfand transform of $\M_\fa$ is completely isometric.
   
   \item The $\rC^*$-envelope of $\M_\fa$ is commutative.
   
   \item There is a closed two-sided ideal $\fJ\subset \rC^*(\M_\fa)$ such that $\rC^*(\M_\fa)/\fJ$ is commutative and the quotient map $q_\fJ:\rC^*(\M_\fa)\to \rC^*(\M_\fa)/\fJ$ is completely isometric on $\M_\fa$.
   
   \item  The identity representation of $\rC^*(\M_\fa)$ is not a boundary representation for $\M_\fa$. 
 \end{enumerate}
Then we have
\[
 {\rm{(i)}} \Leftrightarrow {\rm{(ii)}} \Leftrightarrow {\rm{(iii)}} \Rightarrow {\rm{(iv)}}.
\]
If we assume in addition that $\rC^*(\M_\fa)/\fK(\H_\fa)$ is commutative, then we have
\[
 {\rm{(i)}} \Leftrightarrow {\rm{(ii)}} \Leftrightarrow {\rm{(iii)}} \Leftrightarrow {\rm{(iv)}}.
\]
\end{corollary}
\begin{proof}
This is a special case of Theorem \ref{T:GTC*env} and Corollary \ref{C:GTC*env}. 
\end{proof}

We remark that one reason that the previous statement is not as complete as that of Theorem \ref{T:unifquotientAa} is that the analogue of Theorem \ref{T:Aaessnorm} fails for $\M_\fa$ (see Example \ref{E:cpactquotientFX}). We also note that while it may be tempting, in light of Theorem \ref{T:Toeplitz}, to guess that $\rC^*(\M_\fa)/\fK(\H_\fa)$ is always commutative, this  is unfortunately not the case (see \cite{FX2011}). Corollary \ref{C:Maunifalg} thus raises interesting questions regarding quotients of $\rC^*(\M_\fa)$, which we explore in the following examples.

\begin{example}\label{E:Mdcompactquotient}
Let $d\geq 2$ and let $\H=H^2_d$ be the Drury-Arveson space on $\bB_d$. Let $\fa=\{0\}$ so that $\H=\H_\fa$ and $\M_\fa=\M(\H)$. Then, $\rC^*(\M(\H))/\fK(\H)$ is not commutative by \cite[Theorem 1.2]{FX2011}. Let $q:\rC^*(\M(\H))\to \rC^*(\M(\H))/\fK(\H)$ be the quotient map.  We saw in Example \ref{E:cpactquotientFX} that the Gelfand transforms of $q(\M(\H))$ and $\M(\H)$ are not isometric. Moreover, $q$ is not completely isometric on $\M(\H)$ as explained in Example \ref{E:DAAaGelf}. Hence the identity representation of $\rC^*(\M(\H))$ is a boundary representation for $\M(\H)$ by virtue of Theorem \ref{T:bdrythm}. 
 \qed
\end{example}

Next, we consider non-trivial ideals on the Hardy space.

\begin{example}\label{E:HardyMaGelf}
 Let $\H=H^2_1$ be the Hardy space on the unit disc.  As explained in Example \ref{E:HardyAaGelf}, every weak-$*$ closed ideal of $\M(\H)$ is a principal ideal generated by some inner function.  Let $\theta\in \M(\H)$ be such an inner function and let $\fa=\theta\M(\H)$. It was shown in \cite[Theorem 1]{moore1974} that $\M_\fa$ contains a non-zero compact operator, so that the identity representation of $\rC^*(\M_\fa)$ is a boundary representation for $\M_\fa$ by Theorem \ref{T:bdrythm}. Moreover, the Gelfand transform of $\M_\fa$ is never isometric when $\H_\fa$ has dimension greater than one, by \cite[Theorem 4.2]{clouatre2015CB}.
\qed
\end{example}

We make an elementary remark.
Let $n\geq 1$ and let $A=[\phi_{ij}(Z_\fa)]_{ij}\in \bM_n(\A_\fa)$. We see that
\begin{align*}
 \max\{\|\chi^{(n)}(A)\|_{\bM_n}:\chi\in \Delta(\A_\fa)\}&=\max\{\|[\phi_{ij}(\lambda)]_{ij}\|_{\bM_n}:\lambda\in \widehat{\supp \fa}\}\\
 &\geq \max\{\|[\phi_{ij}(\lambda)]_{ij}\|_{\bM_n}:\lambda\in \supp \fa\}\\
 &= \max\{\|\chi^{(n)}(A)\|_{\bM_n}:\chi\in \Delta(\M_\fa)\}
\end{align*}
by definition of $\supp \fa$ and by Theorem \ref{T:specAaTa}(i). Therefore, if the Gelfand transform of $\M_\fa$ is completely isometric, then so is that of $\A_\fa$. The converse does not generally hold; see Examples \ref{E:HardyAaGelf} and \ref{E:HardyMaGelf}.

The phenomena described in the previous example are not representative of the general multivariate situation, as we illustrate next.

\begin{example}\label{E:DAMaGelf}
Let $\H=H^2_2$ be the Drury--Arveson space on $\bB_2$. Let $\fa=\{0\}$. We saw in Example \ref{E:DAAaGelf} that the Gelfand transform of $\A_\fa$ is not completely isometric, and so neither is that of $\M_\fa$ by the remark above.

Next, let $\fb\subset \M(\H)$ be the weak-$*$ closed ideal generated by $z_2$.  In this case, $\M_\fb$ is unitarily equivalent to $\M(H^2_1)$, which has a completely isometric Gelfand transform as seen in Example \ref{E:HardyAaGelf}.
\qed
\end{example}

We close this section by giving an example where an ideal $\fJ$ fitting into the framework of Corollary \ref{C:Maunifalg} can be identified.

\begin{example}\label{E:commHinf}
Let $\H=H^2_1$ be the Hardy space on the unit disc. Let $\fa=\{0\}$ so that $\H=\H_\fa$ and $\M_\fa=\M(\H)$. In that case, there is a unital completely isometric isomorphism $\Theta:\M(\H)\to H^\infty(\bD)$ such that
\[
 \Theta(M_\psi)=\psi, \quad\psi\in \M(\H)
\]
(see Example \ref{E:HardyAaGelf} for details). Let $\fW\subset B(\H)$ be the $\rC^*$-algebra generated by the Toeplitz operators with  symbols in $L^\infty(\bT,m)$, so that $\M(\H)\subset \fW$ (see \cite[Chapter 7]{douglas1998} for details). Let $\fC\subset \fW$ denote the commutator ideal of $\fW$. It follows from \cite[Theorem 7.11]{douglas1998} that there is a $*$-isomorphism  $\rho:\fW/\fC\to L^\infty(\bT,m)$ such that
\[
 \rho(M_\psi+\fC)=\psi, \quad \psi\in \M(\H).
\]
Next, let $\fJ=\fC\cap \rC^*(\M(\H))$. A standard theorem for $\rC^*$-algebras \cite[Corollary II.5.1.3]{blackadar2006} implies that there is a $*$-isomorphism 
\[
\Xi:\rC^*(\M(\H))/\fJ\to ( \rC^*(\M(\H))+\fC)/\fC
\]
such that
\[
 \Xi(M_\psi+\fJ)=M_\psi+\fC, \quad \psi\in \M(\H).
\]
Let $q_\fJ:\rC^*(\M(\H))\to \rC^*(\M(\H))/\fJ$ be the quotient map. We then note that 
\[
 q_\fJ|_{\M(\H)}=\Xi^{-1}\circ \rho^{-1}\circ \Theta
 \]
 so that $q_\fJ$ is completely isometric on $\M(\H)$. Moreover, $\fW/\fC$ is commutative by choice of $\fC$, and thus so is 
 \[
  \rC^*(\M(\H))/\fJ=\Xi^{-1}( ( \rC^*(\M(\H))+\fC)/\fC).
  \]
Applying Corollary \ref{C:Maunifalg}, we can now recover the classical fact that the Gelfand transform of $\M(\H)$ is completely isometric and that the identity representation of $\rC^*(\M(\H))$ is not a boundary representation for $\M(\H)$.
 \qed
\end{example}

It is known that the ideal $\fJ$ above contains the compact operators \cite[Proposition 7.12]{douglas1998}, but the containment is proper. Indeed, it is readily seen that $M_\theta$ is not essentially normal whenever $\theta$ is an inner function but not a finite Blaschke product, so in particular $\rC^*(\M(\H))/\fK(\H)$ is not commutative.


\section{The Gelfand transform of $\A_\fa$ often fails to be completely isometric}\label{S:polyapprox}

In the previous section, we saw that the determining whether the identity representation of $\fT_\fa$ is a boundary representation for $\A_\fa$ is equivalent to deciding whether the Gelfand transform of $\A_\fa$ fails to be completely isometric. In this section, we seek concrete conditions that address the latter property.

First, we improve on part of \cite[Theorem 8.8]{CH2018}, removing the essential normality condition that was required therein. Recall that a weak-$*$ closed ideal $\fa\subset \M(\H)$ is said to be \emph{homogeneous} if it is the weak-$*$ closure of a polynomial ideal generated by some collection of homogeneous polynomials.

\begin{theorem}\label{T:homogquotientGelf}
 Let $\H$ be a maximal regular unitarily invariant complete Nevanlinna--Pick space. Let $\fa\subset \M(\H)$ be a proper weak-$*$ closed homogeneous ideal such that $\H_\fa$ has dimension greater than one. If $\A_\fa$ is not completely isometrically isomorphic to the disc algebra $A(\bD)$, then the Gelfand transform of $\A_\fa$ is not completely isometric.
\end{theorem}
\begin{proof}
By Theorems \ref{T:bdrythm} and \ref{T:unifquotientAa}, we must show that the quotient map $q:\fT_\fa\to \O_\fa$ is not completely isometric on $\A_\fa$. For $A=[\phi_{ij}(Z_\fa)]_{ij}\in \bM_n(\A_\fa)$, it follows from Theorems \ref{T:specAaTa} and \ref{T:Aaessnorm} that
\begin{align*}
 \|q^{(n)}(A)\|&=\max\{ \|\chi^{(n)}(A)\|_{\bM_n}: \chi\in \Delta(\fT_\fa)\}\\
 &\leq \max\{ \|[\phi_{ij}(\zeta)]_{ij}\|_{\bM_n}: \zeta\in \Z_{\ol{\bB_d}}(\fa \cap \A(\H))\cap \bS_d\}\\
 &\leq \|A\|.
\end{align*}
Since we assume that $\A_\fa$ is not completely isometric to the disc algebra, we may now argue as in the proof of \cite[Theorem 8.8]{CH2018} to see that $q$ is not completely isometric on $\A_\fa$.
\end{proof}

Up to now, our attention has been mostly devoted to maximal spaces. The next result completely settles the non-maximal case.

\begin{theorem}\label{T:nonmax}
 Let $\H$ be a  regular unitarily invariant complete Nevanlinna--Pick space which is not maximal. Let $\fa\subset \M(\H)$ be a proper weak-$*$ closed ideal such that $\H_\fa$ has dimension greater than one. Then, the Gelfand transform of $\A_\fa$ is not completely isometric.
 \end{theorem}
 \begin{proof}
 By assumption, there is a sequence $(b_n)$ of non-negative numbers such that $\sum_{n=1}^\infty b_n<1$ and with the property that
 \[
1-\frac{1}{k(z,w)}= \sum_{n=1}^\infty b_n\langle z,w\rangle^n, \quad z,w\in \bB_d.
 \]
 For convenience, put $r=\sum_{n=1}^\infty b_n.$ For every $\lambda\in \ol{\bB_d}$, we see that
 \[
 \sum_{n=1}^\infty b_n \|\lambda\|^{2n}\leq r<1.
 \]
 Consider the row operator $R=\left(b_{|\alpha|}^{1/2} \left(\frac{|\alpha|!}{\alpha!}\right)^{1/2}Z_\fa^\alpha \right)_{|\alpha|\geq 1}$. It follows from Lemma \ref{L:irred}(ii) that 
 \[
 RR^*=I-P_{\H_\fa}P_{\bC 1}|_{\H_\fa}.
 \]
 Since $\H_\fa$ has dimension at least two, we can find two orthogonal unit vectors $f,g\in \H_\fa$. If $f(0)\neq 0$, then we may consider the non-zero vector
 \[
 h= g-\frac{g(0)}{f(0)}f\in \H_\fa
 \]
 which satisfies $h(0)=0$. Thus, we conclude that $\H_\fa$ always contains a non-zero vector $h$ such that $P_{\bC 1} h=0$. We find
 \[
 RR^*h=(I-P_{\H_\fa}P_{\bC 1}|_{\H_\fa})h=h
 \]
 so that $\|R\|=1$. Hence, there is $N\geq 1$ such that the truncated row \[R_N=\left(b_{|\alpha|}^{1/2} \left(\frac{|\alpha|!}{\alpha!}\right)^{1/2}Z_\fa^\alpha \right)_{1\leq |\alpha|\leq N}\] satisfies $ \|R_N\|>\sqrt{r}$. On the other hand, given $ \chi\in \Delta(\A_\fa)$, by Theorem \ref{T:specAaTa}(ii) there is $\lambda\in \Z_{\ol{\bB_d}}(\fa\cap \A(\H))$ such that
 \begin{align*}
\left\|  \left(\chi\left(b_{|\alpha|}^{1/2} \left(\frac{|\alpha|!}{\alpha!}\right)^{1/2}Z_\fa^\alpha\right) \right)_{1\leq |\alpha|\leq N}\right\|^2&=\sum_{n=1}^N b_n \sum_{|\alpha|=n} \frac{n!}{\alpha!} |\lambda^\alpha|^2\\
&\leq  \sum_{n=1}^\infty b_n \|\lambda\|^{2n}\leq r.
 \end{align*}
 We conclude that Gelfand transform of $\A_\fa$ is not completely isometric.
 \end{proof}

Next, we explore ``size'' restrictions on the boundary portion of the support of $\fa$.
Our approach hinges on the following fact, which generalizes a result on page 444 of \cite{NagyRiesz1990}. To simplify the statement, we introduce some terminology. Given a compact subset $K\subset \ol{\bB_d}$, we say that it is an \emph{approximation set} if the polynomials are uniformly dense in the continuous functions on $K$. Moreover, given a unital contractive representation $\rho$ of $\A(\H)$, we say that $K$ is a \emph{spectral set} for $\rho(\A(\H))$ if 
\[
\|\rho(\phi)\| \leq  \max_{w\in K} |\phi(w)|, \quad \phi\in \A(\H).
\]

\begin{lemma}\label{L:spectralset}
Let $\H$ be a maximal regular unitarily invariant complete Nevanli\-nna--Pick space with kernel $k$.  Let $(b_n)$ be the sequence of non-negative real numbers such that
\[
1-\frac{1}{k(z,w)}=\sum_{n=1}^\infty b_n \langle z,w \rangle^n, \quad z,w\in \bB_d.
\]
Let $\rho:\A(\H)\to B(\E)$ be a unital contractive homomorphism such that $\rho(\A(\H))$ admits an approximation set $K\subset \bS_d$ as a spectral set. Then, $(\rho(z_1),\ldots,\rho(z_d))$ is a spherical unitary and
\[
\sum_{n=1}^\infty b_n \sum_{|\alpha|=n}\frac{n!}{\alpha!}\rho(z^\alpha) \rho(z^{\alpha})^*= I.
\]
\end{lemma}
\begin{proof}
	Let $\P(K)\subset C(K) $ denote the space of polynomials. Because $K$ is a spectral set for $\rho(\A(\H))$,  we find
\[
\|\rho(p)\|\leq \max_{w\in K}|p(w)|, \quad p\in \P(K).
\]
On the other hand, since $K$ is an approximation set, by density we obtain a unital contractive homomorphism $\pi:C(K)\to B(\E)$ such that
\[
\pi(p)=\rho(p), \quad p\in \P(K).
\] 
Thus, $\pi$ must be a $*$-homomorphism \cite[Proposition 2.11]{paulsen2002} and because $K\subset \bS_d$ we see that 
\[
 \sum_{j=1}^d \rho(z_j) \rho(z_j)^*=\pi\left( \sum_{j=1}^d |z_j|^2\right)=I
\]
so that $(\rho(z_1),\ldots,\rho(z_d))$ is a spherical unitary.
Hence
\begin{align*}
\sum_{n=1}^\infty b_n \sum_{|\alpha|=n}\frac{n!}{\alpha!}\rho(z^\alpha) \rho(z^{\alpha})^*&=\sum_{n=1}^\infty b_n \left( \sum_{j=1}^d \rho(z_j) \rho(z_j)^*\right)^n\\
&=\sum_{n=1}^\infty b_n I=I
\end{align*}
where the last equality follows by maximality of $\H$.
\end{proof}

We can now apply the previous theorem to obtain a sufficient condition for the Gelfand transform of $\A_\fa$ not to be completely isometric.

\begin{corollary}\label{C:idbdryPKCK2}
Let $\H$ be a maximal regular unitarily invariant complete Nevanli\-nna--Pick space. Let $\fa\subset \M(\H)$ be a proper weak-$*$ closed ideal such that $\H_\fa$ has dimension greater than one and $\supp \fa\cap \bS_d$ is an approximation set. Then, the Gelfand transform of $\A_\fa$ is not completely isometric.
\end{corollary}
\begin{proof}
Assume that the quotient map $q:\fT_\fa\to \O_\fa$ is completely isometric on $\A_\fa$. By Theorems \ref{T:specAaTa}(iv) and \ref{T:Aaessnorm}, we see that  see that $\supp \fa\cap \bS_d$ is a spectral set for 
\[
\{\phi(Z_\fa):\phi\in \A(\H)\}.
\]
Applying Lemma \ref{L:spectralset} we find
\[
\sum_{n=1}^\infty b_n \sum_{|\alpha|=n}\frac{n!}{\alpha!} Z_\fa^\alpha Z_\fa^{\alpha *}=I.
\]
But this contradicts Lemma \ref{L:irred}(ii). Thus, $q$ is not completely isometric on $\A_\fa$, and the conclusion follows from Theorems \ref{T:bdrythm} and \ref{T:unifquotientAa}.
\end{proof}

We note that the Gelfand transform of $\A_\fa$ can certainly fail to be completely  isometric without $\supp \fa\cap \bS_d$ being an approximation set: this happens for instance when $\H=H^2_3$ and $\fa=\langle z_1\rangle $. 

There are known concrete sufficient conditions that guarantee that a subset is an approximation set. We refer the reader to the papers quoted in the proof for the precise definitions of the properties appearing in the statement below.

\begin{corollary}\label{C:PKCK}
Let $\H$ be a maximal regular unitarily invariant complete Nevanli\-nna--Pick space. 
Let $\fa\subset \M(\H)$ be a proper weak-$*$ closed ideal such that $\H_\fa$ has dimension greater than one.  Let $K=\supp \fa\cap \bS_d$. Then, the Gelfand transform of $\A_\fa$ is not completely isometric whenever  one of the following conditions holds.
	\begin{enumerate}[{\rm(i)}]
		\item $K\subset \bR^d$.
		\item $K$ is polynomially convex and has zero $2$-Hausdorff measure.
		\item $K$ is a polynomially convex real analytic subvariety of $\bC^d$.
		\item $K$ is totally null for the ball algebra.
	\end{enumerate}
\end{corollary}
\begin{proof}
In view of Corollary \ref{C:idbdryPKCK2}, it suffices to show that $K$ is an approximation set whenever one of these conditions hold. See \cite[pages 95, 97 and 116]{levenberg2006} for (i),(ii) and (iii)  and see  \cite[Theorem 10.1.2]{rudin2008} for (iv). 
\end{proof}

We close this section with a generalization of results from \cite{kennedyshalit2015corr} and \cite{kennedyshalit2015}. Therein, geometric conditions on $\Z_{\bB_d}(\fa)$ are given that ensure that the identity representation of $\fT_\fa$ is a boundary representation for $\A_\fa$.  Our generalization will lead to a refinement of Corollary \ref{C:idbdryessnorm}, thus bringing the resulting statement closer to that found in \cite{kennedyshalit2015corr}. 
Though the argument we present here is closely modeled on that of \cite{kennedyshalit2015corr},
we sketch the details for the benefit of the reader so that the required modifications of the argument are evident. 

For $0<\nu\leq 1$, we define a kernel $k^{(\nu)}$ on $\bB_d$ by the formula
\begin{align*}
	k^{(\nu)}(z,w)&=\frac{1}{(1-\ip{z,w})^\nu}\\&=1+\nu\ip{z,w}+\sum_{n=2}^\infty \frac{\nu(\nu+1)\cdots(\nu+n-1)}{n!}\ip{z,w}^n.
\end{align*}
For the remainder of the section, we will focus entirely on the kernels $k^{(\nu)}$ and on the associated reproducing kernel Hilbert spaces of holomorphic functions $\H_\nu$. We make some preliminary observations.

It follows from \cite[Lemma 2.3 and Section 4]{hartz2017isom}  that among maximal regular unitarily invariant completely Nevanlinna--Pick spaces on $\bB_d$, the scale $\{\H_\nu: 0<\nu\leq 1\}$ consists precisely of those that are well behaved with respect to composition with automorphisms of the ball. More precisely, if  $\Theta:\bB_d\to\bB_d$ is a conformal automorphism, then there is a unitary operator $V_\Theta:\H_\nu\to \H_\nu$ such that
\[
 V_\Theta M_\psi=M_{\psi\circ \Theta}V_\Theta, \quad \psi\in \M(\H_\nu).
\]
For each $1\leq j\leq d$, we let $\theta_j:\bB_d\to\bC$ be corresponding component of $\Theta$, so that 
\[
 \Theta=(\theta_1,\ldots,\theta_d).
 \]
It follows from \cite[Theorem 2.2.5]{rudin2008} and \cite[Theorem 2.10]{CT2019spectrum} that $\theta_j\in \A(\H_\nu)$. We also find that $V_\Theta M_{x_j} V_\Theta^*=M_{\theta_j}$. 

Next, let $\fa\subset \M(\H_\nu)$ be a proper weak-$*$ closed ideal. If we let $\fb=V_\Theta \fa V_\Theta^*$, then we see that $\fb$ is another proper weak-$*$ closed ideal. Moreover, it is readily verified that $V_\Theta \H_\fa=\H_\fb$. Consider the unitary operator $W=V_\Theta|_{\H_\fa}:\H_\fa\to\H_\fb$. Writing $Z_\fa=(\fz_1,\ldots,\fz_d)$, we compute that $W\fz_j W^*=\theta_j(Z_\fb)$ for every $1\leq j\leq d$. In turn, this implies that $W\A_\fa W^*=\A_\fb$. We record a useful technical fact.

\begin{lemma}
Let $0<\nu\leq 1$ and let $\fa\subset \M(\H_\nu)$ be a proper weak-$*$ closed ideal. Let $\Theta:\bB_d\to\bB_d$ be a conformal automorphism. Put  $\fb=V_\Theta \fa V_\Theta^*$ and $\lambda= \Theta(0)$. Then,
\[
 \left(\bC (z_1-\lambda_1)+\ldots+\bC (z_d-\lambda_d)\right)\cap \fa=\{0\}
\]
if and only if
\[
 \left(\bC z_1+\ldots+\bC z_d\right)\cap \fb=\{0\}.
\]
\label{L:KSL2}
\end{lemma}
\begin{proof}
Let $r\in  \left(\bC (z_1-\lambda_1)+\ldots+\bC (z_d-\lambda_d)\right)\cap \fa$. Thus, there are $c_1,\ldots,c_d\in \bC$ such that $r=\sum_{j=1}^d c_j(z_j-\lambda_j)$. As before, we write $\Theta=(\theta_1,\ldots,\theta_d)$ with $\theta_1,\ldots,\theta_d\in \A(\H_\nu)$.
By definition of $\fb$, we see that
$\sum_{j=1}^dc_j(\theta_j-\lambda_j)=r\circ \Theta\in \fb.$
Now, by \cite[Theorem 2.2.5]{rudin2008}, there is a $\mu\in\bB_d$ such that the function $(1-\ip{z,\mu})\theta_j$ is a degree-one polynomial for every $1\leq j\leq d$. We infer that
\[
 (1-\ip{z,\mu}) (r\circ \Theta)=\sum_{j=1}^d c_j ((1-\langle z,\mu \rangle )\theta_j -\lambda_j (1-\langle z,\mu \rangle ))
\]
is a degree-one polynomial in $\fb$. On the other hand, because $r$ vanishes at $\lambda$, this degree-one polynomial must vanish at $0$, and thus must lie in $ \left(\bC z_1+\ldots+\bC z_d\right)\cap \fb.$  The converse can be proved similarly.
\end{proof}

Next, we need a standard generalization of the fact that a function whose derivative is injective at a point must itself be injective near that point. In what follows, given a holomorphic function $\phi:\bB_d\to \bC$, we view its derivative $\phi'(0)$ as a linear map from $\bC^d$ to $\bC$.

\begin{lemma}
Let $\A$ be an algebra of holomorphic functions on $\bB_d$. Let $\J\subset \A$ be an ideal such that $0\in\Z_{\bB_d}(\J)$ and
$
 \bigcap_{\phi\in \J} \ker \phi'(0)=\{0\}.
$
 Then,  $0$ is an isolated point of $\Z_{\bB_d}(\J)$.
	\label{L:KSL1}
\end{lemma}
\begin{proof}
	Let $\G$ denote the ring of germs of analytic functions at $0$. Given a function $g$ analytic on a neighborhood of $0$, we let $[g]$ denote the corresponding germ.
	Because $\G$ is Noetherian, the ideal $\ip{[\phi]:\phi\in\J}$ is generated by $[\phi_1],\ldots,[\phi_m]$ for some $\phi_1,\ldots,\phi_m\in \J$ \cite[Theorem II.E.3]{GunningRossi}.

Suppose $v\in\bC^d$ lies in $\bigcap_{j=1}^m \ker \phi_j'(0)$. Let $\phi \in\J$. Correspondingly, choose $g_1,\ldots,g_d$ analytic on a neighborhood of $0$ such that $[\phi]=\sum_{j=1}^m[g_j][\phi_j]$.
	As $0\in \Z_{\bB_d}(\J)$, we find $\phi_1(0)=\ldots=\phi_m(0)=0$ and thus
	\[
	\phi'(0)v = \sum_{j=1}^m g_j(0)\phi_j'(0)v=0. \]
	By assumption, we infer that $v=0$. If we let $D:\bC^d\to \bC^m$ be the matrix defined as
	$
	D=\begin{bmatrix} \phi_1'(0)\\ \phi_2'(0)\\ \vdots \\ \phi_m'(0) \end{bmatrix},
	$
	then we see that $D$ is injective. Hence, there is a $d\times m$ complex matrix $C$ such that $CD=I_d$. Let
	$
	\Psi=C \begin{bmatrix} \phi_1 \\ \phi_2\\\vdots\\\phi_m\end{bmatrix}.
	$
	If we write
	$
	\Psi=\begin{bmatrix} \psi_1 \\ \psi_2\\\vdots\\\psi_d\end{bmatrix}
	$
	then we have that $\psi_1,\ldots,\psi_d\in \J$ and 
	$
	\Psi'(0)=  CD=I_d.
	$
	Thus, $\Psi$ is injective near $0$, which implies that $0$ is an isolated point of $\Z_{\bB_d}(\psi_1,\ldots,\psi_d)$, and hence of $\Z_{\bB_d}(\J)$.\end{proof}

The crucial step of the argument can now be taken, following  \cite{kennedyshalit2015corr}.

\begin{theorem}
Let $0<\nu\leq 1$ and let $\fa\subset \M(\H_\nu)$ be a proper weak-$*$ closed ideal. Assume that there is a point $\lambda\in\Z_{\bB_d}(\fa)$ that is not an isolated point of $\Z_{\bB_d}(\fa)$. Let $q:\fT_\fa\to \O_\fa$ denote the quotient map. Then, 
 $q$ is not completely isometric on $\A_\fa$ under either of the following sets of conditions:
\begin{enumerate}[{\rm (i)}]

\item  $0<\nu<1$, or

\item $\nu=1$ and $\left(\bC (z_1-\lambda_1)+\ldots+\bC (z_d-\lambda_d)\right)\cap \fa=\{0\}$.
\end{enumerate}
In particular, the Gelfand transform of $\A_\fa$ is not completely isometric under these conditions.
	\label{T:KST}
\end{theorem}
\begin{proof}
Because $\lambda$ is not an isolated point of $\Z_{\bB_d}(\fa)$, there are infinitely many points in $\Z_{\bB_d}(\fa)$. In particular, the infinite set $\{k_\mu:\mu\in \Z_{\bB_d}(\fa)\}$ is linearly independent in $\H_\fa$, so that $\H_\fa$ is infinite-dimensional.
Therefore, we focus on showing that $q$ is not completely isometric on $\A_\fa$, since the statement about the Gelfand transform follows immediately from this, by virtue of Theorems \ref{T:bdrythm} and \ref{T:unifquotientAa}. 

We start with a reduction. Let $\Theta:\bB_d\to\bB_d$ be a conformal automorphism such that $\Theta(0)=\lambda$. Let $\fb=V_\Theta \fa V_\Theta^*$. We observed in the discussion preceding Lemma \ref{L:KSL2} that $\A_\fa$ and $\A_\fb$ are unitarily equivalent, so it is equivalent to show that the quotient map $\fT_\fb\to \O_\fb$ is not completely isometric on $\A_\fb$. It follows from Lemma \ref{L:KSL2} that we may as well assume that $\lambda=0$.

We make one more reduction before starting the proof. Invoke Lemma \ref{L:KSL1} to infer the existence of a unit vector $v\in \bigcap_{\phi\in \fa}\ker \phi'(0)$. Choose a constant unitary operator $U:\bC^d\to\bC^d$ such that $Ue_1=v$, where $e_1=(1,0,\ldots,0)\in \bC^d$. Let $\Omega:\bB_d\to\bB_d$ be the conformal automorphism defined as $\Omega(z)=Uz$. Let $\fc=V_\Omega \fa V_\Omega^*$. Then, we see that $e_1\in \bigcap_{\phi\in \fc}\ker \phi'(0)$. Arguing as in the previous paragraph, we can thus assume that this last statement is true of the ideal $\fa$ itself.

Having made these reductions, we start the argument. We note first that because $0\in \Z_{\bB_d}(\fa)$ and $e_1\in \bigcap_{\phi\in \fa}\ker \phi'(0)$, we have $z_1\in [\fa \H_\nu]^\perp =\H_\fa$. Also, it follows from \cite[Remark 7.1(iv)]{hartz2017isom} that $\|z_1\|^2=1/\nu$.
Next, because $0\in \Z_{\bB_d}(\fa)$ we see that $1\in \H_\fa$ so that
	\[ \|\fz_1\|^2\geq \|\fz_1 1\|^2= \|z_1\|^2=1/\nu \]
where we recall that we use the notation $Z_\fa=(\fz_1,\ldots,\fz_d)$. 

In case (i), we see that $\|\fz_1\|>1$. On the other hand, it readily follows from Theorems \ref{T:specAaTa} and \ref{T:Aaessnorm} that $\|q(\fz_1)\|\leq 1$, so that $q$ is not even isometric on $\A_\fa$ in this case. 

For the rest of the proof, we focus on case (ii), and so we assume that $\nu=1$ and that $\left(\bC z_1+\ldots+\bC z_d\right)\cap \fa=\{0\}$. The argument here is identical to that of \cite{kennedyshalit2015corr}. The fact that $z_2,\ldots,z_d\nin \fa$ implies that $\fz_2 1,\ldots,\fz_d 1$ are all non-zero vectors in $\H_\fa$ by Theorem \ref{T:lattice}. Let $C$ be the column operator $(\fz_1,\ldots,\fz_d)$.
	Then
	\[ \|C\|^2\geq \sum_{j=1}^d\|\fz_j1\|^2 =1/\nu+\sum_{j=2}^d \|\fz_j 1\|^2>1/\nu\geq 1. \]
	Once again, Theorems \ref{T:specAaTa} and \ref{T:Aaessnorm} show that the essential norm of $C$ is at most $1$, so indeed $q|_{\A_\fa}$ is not completely isometric in this case as well.
\end{proof}

We can now refine Corollary \ref{C:idbdryessnorm} in some special cases.

\begin{corollary}\label{C:KSMainC}
    Let $0<\nu\leq 1$ and let $\fa\subset \M(\H_\nu)$ be a proper weak-$*$ closed ideal. Assume that there is a point $\lambda\in\Z_{\bB_d}(\fa)$ that is not an isolated point of $\Z_{\bB_d}(\fa)$. Assume also that either 
\begin{enumerate}[{\rm (i)}]

\item  $0<\nu<1$, or

\item $\nu=1$ and $\left(\bC (z_1-\lambda_1)+\ldots+\bC (z_d-\lambda_d)\right)\cap \fa=\{0\}$.
\end{enumerate} 
	Then,
	$\A_\fa$ is hyperrigid in $\fT_\fa$ if and only if $Z_\fa$ is essentially normal.
\end{corollary}
\begin{proof}
Combine Theorem \ref{T:KST} with Theorem \ref{T:bdrythm} and Corollary \ref{C:idbdryessnorm}.
\end{proof}


\section{Isolated points in the character space}\label{S:isol}

In this final section, much in the spirit of Section \ref{S:polyapprox}, we seek concrete conditions that detect whether or not the Gelfand transform is completely isometric. More precisely, we aim to leverage the presence of isolated points in the character space. We start with the norm-closed algebra $\A_\fa$, where the situation is more transparent.

\begin{theorem}\label{T:shilovisolAa}
Let $\H$ be a  regular unitarily invariant complete Nevanlinna--Pick space. Let $\fa\subset \M(\H)$ be a proper weak-$*$ closed ideal. Assume that $\widehat{\supp \fa}$ consists of more than one point and has an isolated point $\lambda\in \bB_d$.
Then, the Gelfand transform of $\A_\fa$ is not isometric.
\end{theorem}

\begin{proof}
By Theorem \ref{T:specAaTa}(i), there is a unique character $\tau_\lambda\in \Delta(\A_\fa)$ such that $\tau_\lambda(Z_\fa)=\lambda$, and this character is easily seen to be an isolated point of $\Delta(\A_\fa)$. 
By the Shilov idempotent theorem \cite[Corollary III.6.5]{gamelin1969}, there is an idempotent element $b\in \A_\fa$ such that $\tau_\lambda(b)=1$ and $\gamma(b)=0$ for every $\gamma\in \Delta(\A_\fa)\setminus\{\tau_\lambda\}$. An approximation argument then yields a multiplier $\theta\in \A(\H)$ such that  $|\theta(\lambda)|>1/2$ and $|\gamma(\theta(Z_\fa))|<1/2$ for every $\gamma\in \Delta(\A_\fa)\setminus\{\tau_\lambda\}$. Since $\Delta(\A_\fa)$ is not a mere singleton, we infer that $\theta(Z_\fa)$ is not a constant multiple of the identity, so $\|\theta(Z_\fa)\|>|\theta(\lambda)|$ by Theorem \ref{T:maximummodulus}. Hence,
\[
 \|\theta(Z_\fa)\|>\max\{|\chi(\theta(Z_\fa))|:\chi\in \Delta(\A_\fa) \}
\]
and the Gelfand transform of $\A_\fa$ is not isometric.
\end{proof}

Before giving a sufficient condition in order for the previous result to be applicable, we need a preliminary observation. Recall that a sequence $\Lambda\subset \bB_d$ is said to be \emph{interpolating} for $\M(\H)$ if the restriction map $\rho:\M(\H)\to \ell^\infty(\Lambda)$ defined as
\[
\rho(\psi)=\psi|_\Lambda, \quad \psi\in\M(\H)
\]
is surjective. These sequences have recently been characterized in \cite{AHMR2019interp}. We will say that a subset $\Omega\subset \ol{\bB_d}$ is a \emph{zero set} for $\A(\H)$ if there is a closed ideal $\J\subset \A(\H)$ such that $\Omega=\Z_{\ol{\bB_d}}(\J)$.

\begin{lemma}\label{L:interppolyconvex}
Let $\H$ be the Drury--Arveson space on $\bB_d$ and let $\Lambda\subset \bB_d$ be an interpolating sequence for $\M(\H)$. Assume that $\ol{\Lambda}$ is a zero set for $\A(\H)$. 
Let 
$
\fa=\{\psi\in \M(\H):\psi|_\Lambda=0\}.
$
Then, $\supp\fa=\ol{\Lambda}$ and thus $\supp \fa$ is polynomially convex.
\end{lemma}
\begin{proof}
We know that
\[
\supp \fa=\{\chi(\fz_1,\ldots,\fz_d):\chi\in \Delta(\M_\fa)\}
\]
so that $\supp \fa$ consists of those points $\lambda=(\lambda_1,\ldots,\lambda_d)\in \bC^d$ such that
\[
(\fz_1-\lambda_1 I)\M_\fa+\ldots+(\fz_d-\lambda_d I)\M_\fa\neq \M_\fa.
\]
It then readily follows from \cite[Corollary 4.7]{CT2019interp} that $\supp\fa=\ol{\Lambda}$. Thus, $\supp \fa$ is a zero set for $\A(\H)$, and hence is automatically polynomially convex.
\end{proof}

We can now give an application of Theorem \ref{T:shilovisolAa}.
Recall that a closed subset $K\subset \bS_d$ is said to be \emph{$\A_d$-totally null} if $|\mu|(K)=0$ for every $\A_d$-Henkin measure $\mu$ on $\bS_d$ (see \cite{CD2016duality} for details).

\begin{corollary}\label{C:isolpointDA}
Let $\H$ be the Drury--Arveson space on $\bB_d$ and let $\Lambda\subset \bB_d$ be an interpolating sequence for $\M(\H)$ consisting of more than one point. Assume that $\ol{\Lambda}\cap \bS_d$ is $\A_d$-totally null. Let 
$
\fa=\{\psi\in \M(\H):\psi|_\Lambda=0\}.
$
Then, the Gelfand transform of $\A_\fa$ is not  isometric.
\end{corollary}
\begin{proof}
It follows from \cite[Corollary 5.13]{CD2018ideals} that $\ol{\Lambda}$ is a zero set for $\A(\H)$. Apply Lemma \ref{L:interppolyconvex} to see that $\widehat{\supp \fa}=\supp \fa=\ol{\Lambda}$. The proof is complete upon applying Theorem \ref{T:shilovisolAa}, as an interpolating sequence cannot have an accumulation in $\bB_d$.
\end{proof}

For the remainder of the section we will turn out attention to $\M_\fa$. To begin we collect some relevant information about the character space of $\M_\fa$, which is typically more complicated than that of $\A_\fa$.

\begin{lemma}\label{L:isolchar}
 Let $\H$ be a  regular unitarily invariant complete Nevanlinna--Pick space. Let $\fa\subset \M(\H)$ be a proper weak-$*$ closed ideal and let $\lambda\in \Z_{\bB_d}(\fa)$. The following statements hold.
 \begin{enumerate}[{\rm(i)}]

  \item  There is a unique character $\tau_\lambda\in \Delta(\M_\fa)$ such that 
  \[
 \lambda=(\tau_\lambda(\fz_1),\ldots,\tau_\lambda(\fz_d)).
\]
Furthermore, this character is weak-$*$ continuous and satisfies
\[
 \tau_\lambda(\psi(Z_\fa))=\psi(\lambda), \quad \psi\in \M(\H).
\]

\item The point $\lambda$ is an isolated point of $\Z_{\bB_d}(\fa)$ if and only if $\tau_\lambda$ is an isolated point of $\Delta(\M_\fa)$.
 \end{enumerate}
\end{lemma}
\begin{proof}
(i) Let $\eps_\lambda$  be the weak-$*$ continuous evaluation character on $\M(\H)$ corresponding to $\lambda$. Recall that  the map $\Gamma:\M(\H)\to \M_\fa$ defined as
\[
 \Gamma(\psi)=\psi(Z_\fa), \quad \psi\in \M(\H)
\]
is a weak-$*$ continuous unital surjective homomorphism. Moreover, we have that $\ker \Gamma=\fa$. Because $\lambda\in \Z_{\bB_d}(\fa)$, we see that $\fa\subset \ker \eps_\lambda$. Thus, there is a weak-$*$ continuous character $\tau_\lambda\in \Delta(\M_\fa)$ such that $\tau_\lambda\circ \Gamma=\eps_\lambda$. Then, we have
  \[
 \lambda=(\tau_\lambda(\fz_1),\ldots,\tau_\lambda(\fz_d)).
\]
To show that $\tau_\lambda$ is unique, let $\chi\in \Delta(\M_\fa)$ such that
  \[
 \lambda=(\chi(\fz_1),\ldots,\chi(\fz_d)).
\]
Then, $\chi\circ \Gamma\in \Delta(\M(\H))$ and 
\[
 ((\chi\circ \Gamma)(M_{z_1}),\ldots,(\chi\circ \Gamma)(M_{z_d}))=\lambda.
\]
It follows from \cite[Proposition 8.5]{hartz2017isom} that $\chi\circ \Gamma=\eps_\lambda$, whence $\chi=\tau_\lambda$ since $\Gamma$ is surjective.

(ii) Assume that $\lambda\in\Z_{\bB_d}(\fa)$ is an isolated point. Choose $r>0$ small enough so that $B(\lambda,r)\subset \bB_d$ and
\[
B(\lambda,r)\cap \Z_{\bB_d}(\fa)=\{\lambda\}.
\] 
To see that $\tau_\lambda$ is an isolated point of $\Delta(\M_\fa)$, consider
\[
 U=\{\chi\in \Delta(\M_\fa): \chi(Z_\fa)\in B(\lambda,r)\}
\]
which is a weak-$*$ open neighbourhood of $\tau_\lambda$ in $\Delta(\M_\fa)$. We claim that $U=\{\tau_\lambda\}$. Indeed, let $\chi\in U$ and put $\mu=\chi(Z_\fa)$. Then, we see that $\mu\in \bB_d$ by choice of $r$. But we also have that $\mu\in \supp \fa$ by definition of the support, so that
\[
 \mu\in \supp \fa\cap \bB_d=\Z_{\bB_d}(\fa).
\]
Hence $\mu\in \Z_{\bB_d}(\fa)\cap B(\lambda,r)=\{\lambda\}$. Invoking  (i) we infer that $\tau_\lambda=\chi$.

Conversely, assume that $\tau_\lambda$ is an isolated point of $\Delta(\M_\fa)$. Thus, there are $\psi_1,\ldots,\psi_n\in \M(\H)$ and $\eps>0$ such that
\[
 \{\chi\in \Delta(\M_\fa):|\chi(\psi_j(Z_\fa))-\tau_\lambda(\psi_j(Z_\fa))|<\eps, \quad 1\leq j\leq d\}=\{\tau_\lambda\}.
\]
Choose $\delta>0$ small enough so that if $\mu\in \bB_d$ satisfies $\|\mu-\lambda\|<\delta$, then
\[
 |\psi_j(\mu)-\psi_j(\lambda)|<\eps, \quad 1\leq j\leq d.
\]
By (i), we conclude that if $\mu\in \Z_{\bB_d}(\fa)\cap B(\lambda,\delta)$ then
\[
 |\tau_\mu(\psi_j(Z_\fa))-\tau_\lambda(\psi_j(Z_\fa))|<\eps, \quad 1\leq j\leq d
\]
whence $\tau_\mu=\tau_\lambda$ and $\mu=\lambda$. Therefore $\lambda$ is an isolated point of $\Z_{\bB_d}(\fa)$.
\end{proof}

We mention in passing one basic consequence of the previous lemma. Roughly speaking, it says that if the Gelfand transform of $\M_\fa$ is isometric, then the support of $\fa$ cannot be contained in the open ball. This will be useful below.

\begin{lemma}\label{L:DeltaMainterior}
Let $\H$ be a  regular unitarily invariant complete Nevanlinna--Pick space. Let $\fa\subset \M(\H)$ be a proper weak-$*$ closed ideal such that $\H_\fa$ has dimension greater than one and $\supp \fa\subset \bB_d$. Then, the Gelfand transform of $\M_\fa$ is not isometric.
\end{lemma}
\begin{proof}
Note that if $\fa+\bC 1=\M(\H)$, then $\fa$ is maximal and thus $\H_\fa$ is one-dimensional by Lemma \ref{L:properideal}. Hence, we may assume that there is $\psi\in \M(\H)$ that does not lie in $\fa+\bC 1$. It thus follows from Theorem \ref{T:maximummodulus} that $\|\psi(Z_\fa)\|>|\psi(\lambda)|$ for every $\lambda\in \Z_{\bB_d}(\fa)$.

Next, let $\chi\in \Delta(\M_\fa)$ and let $\mu=\chi(Z_\fa).$ By assumption, we see that $\mu\in \supp \fa\subset \bB_d$, so that $\mu\in \Z_{\bB_d}(\fa)$.  But using Lemma \ref{L:isolchar}(i) we find  $\chi(\psi(Z_\fa))=\psi(\mu)$. By the first paragraph, we thus conclude that $|\chi(\psi(Z_\fa))|<\|\psi(Z_\fa)\|$. Since $\chi\in \Delta(\M_\fa)$ was arbitrary, we conclude that the Gelfand transform of $\M_\fa$ is not isometric.
\end{proof}

The main thrust for what is to come is given by the following, which is the analogue of Theorem \ref{T:shilovisolAa} in the context of $\M_\fa$.

\begin{theorem}\label{T:shilovisolMa}
Let $\H$ be a  regular unitarily invariant complete Nevanlinna--Pick space. Let $\fa\subset \M(\H)$ be a proper weak-$*$ closed ideal. Assume that $\supp \fa$ has an isolated point $\lambda\in \bB_d$. Then, the following statements hold.
\begin{enumerate}[{\rm (i)}]
\item The algebra $\M_\fa$ contains a non-zero idempotent finite-rank operator and the identity representation of $\rC^*(\M_\fa)$ is a boundary representation for $\M_\fa$. Moreover, the Gelfand transform of $\M_\fa$ is not completely isometric if $\H_\fa$ has dimension greater than one.

 \item If $\supp \fa$ consists of more than one point, then 
the Gelfand transform of $\M_\fa$ is not isometric. 
\end{enumerate}
\end{theorem}
\begin{proof}
(i) By \cite[Theorem 3.8]{CT2019spectrum}, we know that the Taylor spectrum of $Z_\fa$ coincides with
\[
 \{\chi(Z_\fa):\chi\in \Delta(\M_\fa)\}=\supp \fa.
\]
By assumption, $\lambda$ is an isolated point of $\supp \fa$, and hence of the Taylor spectrum of $Z_\fa$. Using \cite[Theorem 2.1(iii)]{CT2019spectrum}, we obtain a non-zero idempotent $E\in \{Z_\fa\}''$ such that the Taylor spectrum of $Z_\fa|_{\ran E}$ consists only of $\{\lambda\}$. On the other hand, we may invoke \cite[Theorem 3.1 and Lemma 4.14]{CT2019spectrum} to see that there are commuting nilpotent operators $N_1,\ldots,N_d$ on $\ran E$ such that 
\[
 \fz_j|_{\ran E}=\lambda_j I_{\ran E}+N_j, \quad 1\leq j\leq d.
\]
Now, since $P_{\H_\fa}1$ is cyclic for $Z_\fa$, the vector $EP_{\H_\fa}1$ is cyclic for $Z_\fa|_{\ran E}$, and it is readily seen that this forces $\ran E$ to be finite-dimensional. Finally, we may apply the Commutant Lifting Theorem \cite[Theorem 5.1]{BTV2001} to obtain $\omega\in \M(\H)$ such that $E=\omega(Z_\fa)$. Then, $\omega(Z_\fa)$ has finite-rank. By virtue of Theorem \ref{T:bdrythm}, we conclude that the identity representation of $\rC^*(\M_\fa)$ is a boundary representation for $\M_\fa$. Applying Corollary \ref{C:Maunifalg}, we also see that 
the Gelfand transform of $\M_\fa$ is not completely isometric whenever $\H_\fa$ has dimension greater than one.

(ii) To see that the Gelfand transform of $\M_\fa$ is not isometric when $\supp \fa$ consists of more than one point, we argue as in the proof of Theorem \ref{T:shilovisolAa}. Let $\tau_\lambda\in \Delta(\M_\fa)$ be as in Lemma \ref{L:isolchar}(i). By Lemma \ref{L:isolchar}(ii), we see that $\tau_\lambda$ is an isolated point of $\Delta(\M_\fa)$. By the Shilov idempotent theorem, there is a multiplier $\theta\in \M(\H)$ such that $\theta(\lambda)=1$ and $\gamma(\theta(Z_\fa))=0$ for every $\gamma\in \Delta(\M_\fa)\setminus\{\tau_\lambda\}$. Since $\Delta(\M_\fa)$ is not a singleton, we infer that $\theta(Z_\fa)$ is not a constant multiple of the identity, so $\|\theta(Z_\fa)\|>|\theta(\lambda)|$ by Theorem \ref{T:maximummodulus}. Hence,
\[
 \|\theta(Z_\fa)\|>\max\{|\chi(\theta(Z_\fa))|:\chi\in \Delta(\M_\fa) \}
\]
and the Gelfand transform of $\M_\fa$ is not isometric. 
                                                                                                                                                                                                                                                                                                                                                  
 \end{proof}

We note that the topological restrictions on the support of the ideal in Theorems \ref{T:shilovisolAa} and \ref{T:shilovisolMa} cannot simply be removed entirely as Examples \ref{E:DAAaGelf} and \ref{E:DAMaGelf} illustrate. Furthermore, the condition that $\supp \fa$ has an isolated point that lies in the \emph{interior} of the unit ball is important for the proof Theorem \ref{T:shilovisolMa}(i) to work, as we show next.

\begin{example}\label{E:intcompact}
 Let $\H=H^2_1$ be the Hardy space on the unit disc. Let $\theta\in \M(\H)$ be the singular inner function
 \[
  \theta(z)=\exp\left(\frac{z+1}{z-1} \right), \quad z\in \bD.
 \]
Let $\fa=\theta\M(\H)$. It is well known that the spectrum of $Z_\fa$ is $\supp \fa=\{1\}$ in this case, so that the multiplier $\omega$ from the proof of Theorem \ref{T:shilovisolMa}(i) can be taken to be the constant function $1$. Thus, $\omega(Z_\fa)=I$ here. But this operator is not compact since $\H_\fa$ is easily seen to be infinite-dimensional. \qed
\end{example}

 As mentioned in Example \ref{E:HardyMaGelf},  in the classical case where $\H$ is the Hardy space on the disc, the Gelfand transform $\M_\fa$ is \emph{never} isometric, unless $\H_\fa$ is one-dimensional \cite[Theorem 4.2]{clouatre2015CB}. We suspect that this a reflection of a more general phenomenon that we will explore further in upcoming work. We substantiate our suspicion in the following result.

\begin{corollary}\label{C:sequenceMa}
	Let $\H$ be a  regular unitarily invariant complete Nevanlinna--Pick space. Let $\fa\subset \M(\H)$ be a proper weak-$*$ closed ideal with the property that $\Z_{\bB_d}(\fa)$ has an isolated point. 
Then, the following statements are equivalent.
 \begin{enumerate}
  \item[\rm{(i)}] The ideal $\fa$ is the vanishing ideal of a single point.
  
  \item[\rm{(ii)}] The space $\H_\fa$ is one-dimensional.
  
    \item[\rm{(iii)}] The Gelfand transform of $\M_\fa$ is completely isometric.
    
    \item[\rm{(iv)}] The Gelfand transform of $\M_\fa$ is  isometric. 
 \end{enumerate}
\end{corollary}

\begin{proof}
(i) $\Rightarrow$ (ii): We assume that there is $\lambda\in \bB_d$ with the property that
\[
 \fa=\{\psi\in \M(\H):\psi(\lambda)=0\}.
\]
 Then, $k_\lambda\in \H_\fa.$ Next, we note that $\bC k_\lambda$ is co-invariant for $\M(\H)$, so by Theorem \ref{T:lattice} there is a weak-$*$ closed ideal $\fb\subset \M(\H)$ such that $\bC k_\lambda =\H_\fb$ and $\fa\subset \fb$. On the other hand, the equality $[\fb \H]=(\bC k_\lambda)^\perp$ implies that every multiplier in $\fb$ vanishes at $\lambda$, whence $\fb \subset \fa$. We conclude that $\H_\fa=\H_\fb=\bC k_\lambda$.

(ii) $\Rightarrow$ (i): It follows from Lemma \ref{L:properideal} that $\fa$ is a maximal ideal, so there is a character $\chi\in \Delta(\M_\fa)$ such that $\fa=\ker \chi$. By assumption, there is $\lambda\in \Z_{\bB_d}(\fa)$. Let $\tau_\lambda\in \Delta(\M_\fa)$ be the corresponding character of evaluation at $\lambda$ (see Lemma \ref{L:isolchar}(i)). Then, we have
\[
 \ker \chi=\fa\subset \ker \tau_\lambda
\]
which forces $\chi=\tau_\lambda$ and $\fa=\ker \tau_\lambda$ as desired.

(ii) $\Rightarrow$ (iii),(iii) $\Rightarrow$ (iv) : This is trivial.

(iv) $\Rightarrow$ (ii): By Theorem \ref{T:shilovisolMa}(ii), we see that $\supp \fa$ consists of only one point lying in $\bB_d$. But then Lemma \ref{L:DeltaMainterior} forces $\H_\fa$ to be one-dimensional.\end{proof}

The previous corollary can be seen as a partial multivariate  analogue of \cite[Theorem 4.2]{clouatre2015CB}.

We close the paper with a result that illustrates that the topological condition of having an isolated zero can be removed entirely in some cases, leveraging our generalization of \cite{kennedyshalit2015corr}.

\begin{corollary}
    Let $0<\nu\leq 1$ and let $\H_\nu$ be the space corresponding to the kernel
    \[
     k^{(\nu)}(z,w)=\frac{1}{(1-\ip{z,w})^\nu}, \quad z,w\in \bB_d.
    \]
Let $\fa$ be a proper weak-$*$ closed ideal of $\M(\H_\nu)$ and let $\lambda\in \Z_{\bB_d}(\fa)$. 
Assume that $\supp \fa$ consists of more than one point and that either 
\begin{enumerate}[{\rm (i)}]

\item  $0<\nu<1$, or

\item $\nu=1$ and $\left(\bC (z_1-\lambda_1)+\ldots+\bC (z_d-\lambda_d)\right)\cap \fa=\{0\}$.
\end{enumerate} 
	Then, the Gelfand transform of $\M_\fa$ is not completely isometric.
	If, in addition, $\supp \fa$ is polynomially convex,
 	then the Gelfand transform of $\A_\fa$ is not completely isometric.
\end{corollary}
\begin{proof}
	First suppose that $\lambda$ is not an isolated point of $\Z_{\bB_d}(\fa)$.
	Then, it follows from Theorem \ref{T:KST} that the Gelfand transform of $\A_\fa$ is not completely isometric. As explained in the remark following Example \ref{E:HardyMaGelf}, the Gelfand transform of $\M_\fa$ is not completely isometric either. 
	
	Next, suppose that $\lambda$ is an isolated point of $\Z_{\bB_d}(\fa)$.
	Since $\supp \fa$ consists of more than one point, it follows from Theorem \ref{T:shilovisolMa}(ii) that the Gelfand transform of $\M_\fa$ is not isometric. If $\supp \fa$ is polynomially convex, then we see that $\lambda$ is also an isolated point of $\widehat{\supp\fa }$, and it now follows from Theorem \ref{T:shilovisolAa} that the Gelfand transformation of $\A_\fa$ is not completely isometric.
\end{proof}

\bibliography{biblio_main}

\def\cprime{$'$}
\begin{bibdiv}
\begin{biblist}

\bib{agler1982}{article}{
      author={Agler, Jim},
       title={The {A}rveson extension theorem and coanalytic models},
        date={1982},
        ISSN={0378-620X},
     journal={Integral Equations Operator Theory},
      volume={5},
      number={5},
       pages={608\ndash 631},
         url={http://dx.doi.org.proxy.lib.uwaterloo.ca/10.1007/BF01694057},
      review={\MR{697007 (84g:47011)}},
}

\bib{AM2000}{article}{
      author={Agler, Jim},
      author={McCarthy, John~E.},
       title={Complete {N}evanlinna-{P}ick kernels},
        date={2000},
        ISSN={0022-1236},
     journal={J. Funct. Anal.},
      volume={175},
      number={1},
       pages={111\ndash 124},
         url={http://dx.doi.org.uml.idm.oclc.org/10.1006/jfan.2000.3599},
      review={\MR{1774853}},
}

\bib{agler2002}{book}{
      author={Agler, Jim},
      author={McCarthy, John~E.},
       title={Pick interpolation and {H}ilbert function spaces},
      series={Graduate Studies in Mathematics},
   publisher={American Mathematical Society, Providence, RI},
        date={2002},
      volume={44},
        ISBN={0-8218-2898-3},
         url={http://dx.doi.org.proxy.lib.uwaterloo.ca/10.1090/gsm/044},
      review={\MR{1882259 (2003b:47001)}},
}

\bib{AHMR2019interp}{article}{
      author={Aleman, Alexandru},
      author={Hartz, Michael},
      author={McCarthy, John~E.},
      author={Richter, Stefan},
       title={Interpolating sequences in spaces with the complete pick
  property},
        date={2019},
        ISSN={1073-7928},
     journal={Int. Math. Res. Not. IMRN},
      number={12},
       pages={3832\ndash 3854},
         url={https://doi-org.uml.idm.oclc.org/10.1093/imrn/rnx237},
      review={\MR{3973111}},
}

\bib{ambrozie2002}{article}{
      author={Ambrozie, C.-G.},
      author={Engli{\v{s}}, M.},
      author={M{\"u}ller, V.},
       title={Operator tuples and analytic models over general domains in
  {$\Bbb C^n$}},
        date={2002},
        ISSN={0379-4024},
     journal={J. Operator Theory},
      volume={47},
      number={2},
       pages={287\ndash 302},
      review={\MR{1911848 (2004c:47013)}},
}

\bib{arveson1969}{article}{
      author={Arveson, William},
       title={Subalgebras of {$C^{\ast} $}-algebras},
        date={1969},
        ISSN={0001-5962},
     journal={Acta Math.},
      volume={123},
       pages={141\ndash 224},
      review={\MR{0253059 (40 \#6274)}},
}

\bib{arveson1972}{article}{
      author={Arveson, William},
       title={Subalgebras of {$C^{\ast} $}-algebras. {II}},
        date={1972},
        ISSN={0001-5962},
     journal={Acta Math.},
      volume={128},
      number={3-4},
       pages={271\ndash 308},
      review={\MR{0394232 (52 \#15035)}},
}

\bib{arveson1976inv}{book}{
      author={Arveson, William},
       title={An invitation to {$C\sp*$}-algebras},
   publisher={Springer-Verlag, New York-Heidelberg},
        date={1976},
        note={Graduate Texts in Mathematics, No. 39},
      review={\MR{0512360}},
}

\bib{arveson1998}{article}{
      author={Arveson, William},
       title={Subalgebras of {$C^*$}-algebras. {III}. {M}ultivariable operator
  theory},
        date={1998},
        ISSN={0001-5962},
     journal={Acta Math.},
      volume={181},
      number={2},
       pages={159\ndash 228},
         url={http://dx.doi.org/10.1007/BF02392585},
      review={\MR{1668582 (2000e:47013)}},
}

\bib{arveson2008}{article}{
      author={Arveson, William},
       title={The noncommutative {C}hoquet boundary},
        date={2008},
        ISSN={0894-0347},
     journal={J. Amer. Math. Soc.},
      volume={21},
      number={4},
       pages={1065\ndash 1084},
         url={http://dx.doi.org/10.1090/S0894-0347-07-00570-X},
      review={\MR{2425180 (2009g:46108)}},
}

\bib{arveson2011}{article}{
      author={Arveson, William},
       title={The noncommutative {C}hoquet boundary {II}: hyperrigidity},
        date={2011},
        ISSN={0021-2172},
     journal={Israel J. Math.},
      volume={184},
       pages={349\ndash 385},
         url={http://dx.doi.org/10.1007/s11856-011-0071-z},
      review={\MR{2823981}},
}

\bib{athavale1990}{article}{
      author={Athavale, Ameer},
       title={On the intertwining of joint isometries},
        date={1990},
        ISSN={0379-4024},
     journal={J. Operator Theory},
      volume={23},
      number={2},
       pages={339\ndash 350},
      review={\MR{1066811}},
}

\bib{BTV2001}{incollection}{
      author={Ball, Joseph~A.},
      author={Trent, Tavan~T.},
      author={Vinnikov, Victor},
       title={Interpolation and commutant lifting for multipliers on
  reproducing kernel {H}ilbert spaces},
        date={2001},
   booktitle={Operator theory and analysis ({A}msterdam, 1997)},
      series={Oper. Theory Adv. Appl.},
      volume={122},
   publisher={Birkh\"{a}user, Basel},
       pages={89\ndash 138},
      review={\MR{1846055}},
}

\bib{bercovici1988}{book}{
      author={Bercovici, Hari},
       title={Operator theory and arithmetic in {$H^\infty$}},
      series={Mathematical Surveys and Monographs},
   publisher={American Mathematical Society},
     address={Providence, RI},
        date={1988},
      volume={26},
        ISBN={0-8218-1528-8},
      review={\MR{954383 (90e:47001)}},
}

\bib{BHM2018}{article}{
      author={Bickel, Kelly},
      author={Hartz, Michael},
      author={McCarthy, John~E.},
       title={A multiplier algebra functional calculus},
        date={2018},
        ISSN={0002-9947},
     journal={Trans. Amer. Math. Soc.},
      volume={370},
      number={12},
       pages={8467\ndash 8482},
         url={https://doi-org.uml.idm.oclc.org/10.1090/tran/7381},
      review={\MR{3864384}},
}

\bib{blackadar2006}{book}{
      author={Blackadar, B.},
       title={Operator algebras},
      series={Encyclopaedia of Mathematical Sciences},
   publisher={Springer-Verlag, Berlin},
        date={2006},
      volume={122},
        ISBN={978-3-540-28486-4; 3-540-28486-9},
         url={https://doi-org.uml.idm.oclc.org/10.1007/3-540-28517-2},
        note={Theory of $C^*$-algebras and von Neumann algebras, Operator
  Algebras and Non-commutative Geometry, III},
      review={\MR{2188261}},
}

\bib{BLM2004}{book}{
      author={Blecher, David~P.},
      author={Le~Merdy, Christian},
       title={Operator algebras and their modules---an operator space
  approach},
      series={London Mathematical Society Monographs. New Series},
   publisher={The Clarendon Press, Oxford University Press, Oxford},
        date={2004},
      volume={30},
        ISBN={0-19-852659-8},
  url={http://dx.doi.org.uml.idm.oclc.org/10.1093/acprof:oso/9780198526599.001.0001},
        note={Oxford Science Publications},
      review={\MR{2111973}},
}

\bib{clouatre2015CB}{article}{
      author={Clou{\^a}tre, Rapha{\"e}l},
       title={Completely bounded isomorphisms of operator algebras and
  similarity to complete isometries},
        date={2015},
        ISSN={0022-2518},
     journal={Indiana Univ. Math. J.},
      volume={64},
      number={3},
       pages={825\ndash 846},
         url={http://dx.doi.org.uml.idm.oclc.org/10.1512/iumj.2015.64.5572},
      review={\MR{3361288}},
}

\bib{clouatre2015jordan}{article}{
      author={Clou{\^a}tre, Rapha{\"e}l},
       title={Unitary equivalence and similarity to {J}ordan models for weak
  contractions of class {$C_0$}},
        date={2015},
        ISSN={0008-414X},
     journal={Canad. J. Math.},
      volume={67},
      number={1},
       pages={132\ndash 151},
         url={http://dx.doi.org.proxy.lib.uwaterloo.ca/10.4153/CJM-2013-044-0},
      review={\MR{3292697}},
}

\bib{clouatre2018lochyp}{article}{
      author={Clou\^{a}tre, Rapha\"{e}l},
       title={Non-commutative peaking phenomena and a local version of the
  hyperrigidity conjecture},
        date={2018},
        ISSN={0024-6115},
     journal={Proc. Lond. Math. Soc. (3)},
      volume={117},
      number={2},
       pages={221\ndash 245},
         url={https://doi-org.uml.idm.oclc.org/10.1112/plms.12133},
      review={\MR{3851322}},
}

\bib{clouatre2018unp}{article}{
      author={Clou\^{a}tre, Rapha\"{e}l},
       title={Unperforated pairs of operator spaces and hyperrigidity of
  operator systems},
        date={2018},
        ISSN={0008-414X},
     journal={Canad. J. Math.},
      volume={70},
      number={6},
       pages={1236\ndash 1260},
         url={https://doi-org.uml.idm.oclc.org/10.4153/CJM-2018-008-1},
      review={\MR{3850542}},
}

\bib{CD2016duality}{article}{
      author={Clou\^atre, Rapha\"el},
      author={Davidson, Kenneth~R.},
       title={Duality, convexity and peak interpolation in the
  {D}rury-{A}rveson space},
        date={2016},
        ISSN={0001-8708},
     journal={Adv. Math.},
      volume={295},
       pages={90\ndash 149},
         url={http://dx.doi.org.uml.idm.oclc.org/10.1016/j.aim.2016.02.035},
      review={\MR{3488033}},
}

\bib{CD2018ideals}{article}{
      author={Clou\^{a}tre, Rapha\"{e}l},
      author={Davidson, Kenneth~R.},
       title={Ideals in a multiplier algebra on the ball},
        date={2018},
        ISSN={0002-9947},
     journal={Trans. Amer. Math. Soc.},
      volume={370},
      number={3},
       pages={1509\ndash 1527},
         url={https://doi-org.uml.idm.oclc.org/10.1090/tran/7007},
      review={\MR{3739183}},
}

\bib{CH2018}{article}{
      author={Clou\^atre, Rapha\"el},
      author={Hartz, Michael},
       title={Multiplier algebras of complete {N}evanlinna-{P}ick spaces:
  dilations, boundary representations and hyperrigidity},
        date={2018},
        ISSN={0022-1236},
     journal={J. Funct. Anal.},
      volume={274},
      number={6},
       pages={1690\ndash 1738},
         url={https://doi-org.uml.idm.oclc.org/10.1016/j.jfa.2017.10.008},
      review={\MR{3758546}},
}

\bib{CT2019cyclic}{article}{
      author={Clou\^{a}tre, Rapha\"{e}l},
      author={Timko, Edward~J.},
       title={Cyclic row contractions and rigidity of invariant subspaces},
        date={2019},
        ISSN={0022-247X},
     journal={J. Math. Anal. Appl.},
      volume={479},
      number={2},
       pages={1906\ndash 1938},
         url={https://doi-org.uml.idm.oclc.org/10.1016/j.jmaa.2019.07.031},
      review={\MR{3987939}},
}

\bib{CT2019interp}{article}{
      author={Clou\^{a}tre, Rapha\"{e}l},
      author={Timko, Edward~J.},
       title={Row contractions annihilated by interpolating vanishing ideals},
        date={2019},
     journal={to appear in International Mathematics Research Notices,
  arXiv:1902.06826},
}

\bib{CT2019spectrum}{article}{
      author={Clou\^{a}tre, Rapha\"{e}l},
      author={Timko, Edward~J.},
       title={The spectra of constrained commuting row contractions},
        date={2019},
     journal={preprint arXiv:1911.03525},
}

\bib{costea2011}{article}{
      author={Costea, Serban},
      author={Sawyer, Eric~T.},
      author={Wick, Brett~D.},
       title={The {C}orona theorem for the {D}rury-{A}rveson {H}ardy space and
  other holomorphic {B}esov-{S}obolov spaces on the unit ball in
  $\mathbb{C}^n$},
        date={2011},
     journal={Anal. PDE},
      volume={4},
      number={4},
       pages={499\ndash 550},
}

\bib{curto1988}{incollection}{
      author={Curto, Ra\'{u}l~E.},
       title={Applications of several complex variables to multiparameter
  spectral theory},
        date={1988},
   booktitle={Surveys of some recent results in operator theory, {V}ol. {II}},
      series={Pitman Res. Notes Math. Ser.},
      volume={192},
   publisher={Longman Sci. Tech., Harlow},
       pages={25\ndash 90},
      review={\MR{976843}},
}

\bib{davidson2015}{article}{
      author={Davidson, Kenneth~R.},
      author={Kennedy, Matthew},
       title={The {C}hoquet boundary of an operator system},
        date={2015},
        ISSN={0012-7094},
     journal={Duke Math. J.},
      volume={164},
      number={15},
       pages={2989\ndash 3004},
         url={http://dx.doi.org.uml.idm.oclc.org/10.1215/00127094-3165004},
      review={\MR{3430455}},
}

\bib{DK2016}{article}{
      author={Davidson, Kenneth~R.},
      author={Kennedy, Matthew},
       title={Choquet order and hyperrigidity for function systems},
        date={2016},
     journal={preprint arXiv:1608.02334},
}

\bib{douglas1998}{book}{
      author={Douglas, Ronald~G.},
       title={Banach algebra techniques in operator theory},
     edition={Second},
      series={Graduate Texts in Mathematics},
   publisher={Springer-Verlag, New York},
        date={1998},
      volume={179},
        ISBN={0-387-98377-5},
         url={https://doi-org.uml.idm.oclc.org/10.1007/978-1-4612-1656-8},
      review={\MR{1634900}},
}

\bib{DTY2016}{article}{
      author={Douglas, Ronald~G.},
      author={Tang, Xiang},
      author={Yu, Guoliang},
       title={An analytic {G}rothendieck {R}iemann {R}och theorem},
        date={2016},
        ISSN={0001-8708},
     journal={Adv. Math.},
      volume={294},
       pages={307\ndash 331},
         url={http://dx.doi.org.uml.idm.oclc.org/10.1016/j.aim.2016.02.031},
      review={\MR{3479565}},
}

\bib{dritschel2005}{article}{
      author={Dritschel, Michael~A.},
      author={McCullough, Scott~A.},
       title={Boundary representations for families of representations of
  operator algebras and spaces},
        date={2005},
        ISSN={0379-4024},
     journal={J. Operator Theory},
      volume={53},
      number={1},
       pages={159\ndash 167},
      review={\MR{2132691 (2006a:47095)}},
}

\bib{EKMR2014}{book}{
      author={El-Fallah, Omar},
      author={Kellay, Karim},
      author={Mashreghi, Javad},
      author={Ransford, Thomas},
       title={A primer on the {D}irichlet space},
      series={Cambridge Tracts in Mathematics},
   publisher={Cambridge University Press, Cambridge},
        date={2014},
      volume={203},
        ISBN={978-1-107-04752-5},
      review={\MR{3185375}},
}

\bib{englis2015}{article}{
      author={Engli{\v{s}}, Miroslav},
      author={Eschmeier, J{\"o}rg},
       title={Geometric {A}rveson-{D}ouglas conjecture},
        date={2015},
        ISSN={0001-8708},
     journal={Adv. Math.},
      volume={274},
       pages={606\ndash 630},
  url={http://dx.doi.org.proxy1.lib.umanitoba.ca/10.1016/j.aim.2014.11.026},
      review={\MR{3318162}},
}

\bib{FX2011}{article}{
      author={Fang, Quanlei},
      author={Xia, Jingbo},
       title={Multipliers and essential norm on the {D}rury-{A}rveson space},
        date={2011},
        ISSN={0002-9939},
     journal={Proc. Amer. Math. Soc.},
      volume={139},
      number={7},
       pages={2497\ndash 2504},
  url={https://doi-org.uml.idm.oclc.org/10.1090/S0002-9939-2010-10680-9},
      review={\MR{2784815}},
}

\bib{gamelin1969}{book}{
      author={Gamelin, Theodore~W.},
       title={Uniform algebras},
   publisher={Prentice-Hall, Inc., Englewood Cliffs, N. J.},
        date={1969},
      review={\MR{0410387}},
}

\bib{GunningRossi}{book}{
      author={Gunning, Robert~C.},
      author={Rossi, Hugo},
       title={Analytic functions of several complex variables},
   publisher={Prentice-Hall, Inc., Englewood Cliffs, N.J.},
        date={1965},
      review={\MR{0180696}},
}

\bib{GHX04}{article}{
      author={Guo, Kunyu},
      author={Hu, Junyun},
      author={Xu, Xianmin},
       title={{Toeplitz} algebras, subnormal tuples and rigidity on reproducing
  {$\bold C[z_1,\dots,z_d]$}-modules},
        date={2004},
        ISSN={0022-1236},
     journal={J. Funct. Anal.},
      volume={210},
      number={1},
       pages={214\ndash 247},
      review={\MR{2052120 (2005a:47007)}},
}

\bib{hamana1979}{article}{
      author={Hamana, Masamichi},
       title={Injective envelopes of operator systems},
        date={1979},
        ISSN={0034-5318},
     journal={Publ. Res. Inst. Math. Sci.},
      volume={15},
      number={3},
       pages={773\ndash 785},
  url={http://dx.doi.org.proxy.lib.uwaterloo.ca/10.2977/prims/1195187876},
      review={\MR{566081 (81h:46071)}},
}

\bib{hartz2017isom}{article}{
      author={Hartz, Michael},
       title={On the isomorphism problem for multiplier algebras of
  {N}evanlinna-{P}ick spaces},
        date={2017},
        ISSN={0008-414X},
     journal={Canad. J. Math.},
      volume={69},
      number={1},
       pages={54\ndash 106},
         url={https://doi-org.uml.idm.oclc.org/10.4153/CJM-2015-050-6},
      review={\MR{3589854}},
}

\bib{kennedyshalit2015}{article}{
      author={Kennedy, Matthew},
      author={Shalit, Orr~Moshe},
       title={Essential normality, essential norms and hyperrigidity},
        date={2015},
        ISSN={0022-1236},
     journal={J. Funct. Anal.},
      volume={268},
      number={10},
       pages={2990\ndash 3016},
         url={http://dx.doi.org.uml.idm.oclc.org/10.1016/j.jfa.2015.03.014},
      review={\MR{3331791}},
}

\bib{kennedyshalit2015corr}{article}{
      author={Kennedy, Matthew},
      author={Shalit, Orr~Moshe},
       title={Corrigendum to ``{E}ssential normality, essential norms and
  hyperrigidity'' [{J}. {F}unct. {A}nal. 268 (2015) 2990--3016] [
  {MR}3331791]},
        date={2016},
        ISSN={0022-1236},
     journal={J. Funct. Anal.},
      volume={270},
      number={7},
       pages={2812\ndash 2815},
         url={https://doi-org.uml.idm.oclc.org/10.1016/j.jfa.2015.12.015},
      review={\MR{3464058}},
}

\bib{kim2019}{article}{
      author={Kim, Se-Jin},
       title={Hyperrigidity of ${C}^*$-correspondences},
        date={2019},
     journal={preprint arXiv:1905.10473},
}

\bib{levenberg2006}{article}{
      author={Levenberg, Norm},
       title={Approximation in {$\Bbb C^N$}},
        date={2006},
        ISSN={1555-578X},
     journal={Surv. Approx. Theory},
      volume={2},
       pages={92\ndash 140},
      review={\MR{2276419}},
}

\bib{moore1974}{article}{
      author={Moore, Berrien, III},
      author={Nordgren, Eric},
       title={On transitive algebras containing {$C_{0}$} operators},
        date={1974/75},
        ISSN={0022-2518},
     journal={Indiana Univ. Math. J.},
      volume={24},
       pages={777\ndash 784},
      review={\MR{0361846 (50 \#14291)}},
}

\bib{muller1993}{article}{
      author={M{\"u}ller, V.},
      author={Vasilescu, F.-H.},
       title={Standard models for some commuting multioperators},
        date={1993},
        ISSN={0002-9939},
     journal={Proc. Amer. Math. Soc.},
      volume={117},
      number={4},
       pages={979\ndash 989},
         url={http://dx.doi.org.proxy.lib.uwaterloo.ca/10.2307/2159525},
      review={\MR{1112498 (93e:47016)}},
}

\bib{muller2007}{book}{
      author={M{\"u}ller, Vladimir},
       title={Spectral theory of linear operators and spectral systems in
  {B}anach algebras},
     edition={Second},
      series={Operator Theory: Advances and Applications},
   publisher={Birkh\"auser Verlag, Basel},
        date={2007},
      volume={139},
        ISBN={978-3-7643-8264-3},
      review={\MR{2355630 (2008g:47013)}},
}

\bib{paulsen2002}{book}{
      author={Paulsen, Vern},
       title={Completely bounded maps and operator algebras},
      series={Cambridge Studies in Advanced Mathematics},
   publisher={Cambridge University Press},
     address={Cambridge},
        date={2002},
      volume={78},
        ISBN={0-521-81669-6},
      review={\MR{1976867 (2004c:46118)}},
}

\bib{PR2016}{book}{
      author={Paulsen, Vern~I.},
      author={Raghupathi, Mrinal},
       title={An introduction to the theory of reproducing kernel {H}ilbert
  spaces},
      series={Cambridge Studies in Advanced Mathematics},
   publisher={Cambridge University Press, Cambridge},
        date={2016},
      volume={152},
        ISBN={978-1-107-10409-9},
         url={https://doi-org.uml.idm.oclc.org/10.1017/CBO9781316219232},
      review={\MR{3526117}},
}

\bib{popescu2006}{article}{
      author={Popescu, Gelu},
       title={Operator theory on noncommutative varieties},
        date={2006},
        ISSN={0022-2518},
     journal={Indiana Univ. Math. J.},
      volume={55},
      number={2},
       pages={389\ndash 442},
  url={http://dx.doi.org.proxy.lib.uwaterloo.ca/10.1512/iumj.2006.55.2771},
      review={\MR{2225440 (2007m:47008)}},
}

\bib{putinar1984}{article}{
      author={Putinar, Mihai},
       title={Spectral inclusion for subnormal {$n$}-tuples},
        date={1984},
        ISSN={0002-9939},
     journal={Proc. Amer. Math. Soc.},
      volume={90},
      number={3},
       pages={405\ndash 406},
         url={https://doi-org.uml.idm.oclc.org/10.2307/2044482},
      review={\MR{728357}},
}

\bib{richter2010}{article}{
      author={Richter, Stefan},
      author={Sundberg, Carl},
       title={Joint extensions in families of contractive commuting operator
  tuples},
        date={2010},
        ISSN={0022-1236},
     journal={J. Funct. Anal.},
      volume={258},
      number={10},
       pages={3319\ndash 3346},
  url={http://dx.doi.org.proxy.lib.uwaterloo.ca/10.1016/j.jfa.2010.01.011},
      review={\MR{2601618 (2011a:47019)}},
}

\bib{NagyRiesz1990}{book}{
      author={Riesz, Frigyes},
      author={Sz.-Nagy, B\'{e}la},
       title={Functional analysis},
      series={Dover Books on Advanced Mathematics},
   publisher={Dover Publications, Inc., New York},
        date={1990},
        ISBN={0-486-66289-6},
        note={Translated from the second French edition by Leo F. Boron,
  Reprint of the 1955 original},
      review={\MR{1068530}},
}

\bib{rudin2008}{book}{
      author={Rudin, Walter},
       title={Function theory in the unit ball of {$\Bbb C^n$}},
      series={Classics in Mathematics},
   publisher={Springer-Verlag, Berlin},
        date={2008},
        ISBN={978-3-540-68272-1},
        note={Reprint of the 1980 edition},
      review={\MR{2446682 (2009g:32001)}},
}

\bib{salomon2019}{article}{
      author={Salomon, Guy},
       title={Hyperrigid subsets of {C}untz-{K}rieger algebras and the property
  of rigidity at zero},
        date={2019},
        ISSN={0379-4024},
     journal={J. Operator Theory},
      volume={81},
      number={1},
       pages={61\ndash 79},
      review={\MR{3920688}},
}

\bib{sarason1967}{article}{
      author={Sarason, Donald},
       title={Generalized interpolation in {$H^{\infty }$}},
        date={1967},
        ISSN={0002-9947},
     journal={Trans. Amer. Math. Soc.},
      volume={127},
       pages={179\ndash 203},
      review={\MR{0208383 (34 \#8193)}},
}

\bib{Nagy1968}{article}{
      author={Sz.-Nagy, B{\'e}la},
      author={Foia{\c{s}}, Ciprian},
       title={Dilatation des commutants d'op\'erateurs},
        date={1968},
     journal={C. R. Acad. Sci. Paris S\'er. A-B},
      volume={266},
       pages={A493\ndash A495},
      review={\MR{0236755 (38 \#5049)}},
}

\bib{nagy2010}{book}{
      author={Sz.-Nagy, B{\'e}la},
      author={Foias, Ciprian},
      author={Bercovici, Hari},
      author={K{\'e}rchy, L{\'a}szl{\'o}},
       title={Harmonic analysis of operators on {H}ilbert space},
     edition={enlarged},
      series={Universitext},
   publisher={Springer},
     address={New York},
        date={2010},
        ISBN={978-1-4419-6093-1},
         url={http://dx.doi.org/10.1007/978-1-4419-6094-8},
      review={\MR{2760647 (2012b:47001)}},
}

\end{biblist}
\end{bibdiv}
\bibliographystyle{plain}

\end{document}